\documentclass[12pt,a4paper]{amsart}
\usepackage{amsmath,amsthm,amssymb}
\usepackage{color}
\usepackage[top=30truemm,bottom=30truemm,left=25truemm,right=25truemm]{geometry}
\usepackage{hyperref}

\newtheorem{thm}{Theorem}[subsection]
\newtheorem{lem}[thm]{Lemma}
\newtheorem{prop}[thm]{Proposition}
\newtheorem{cor}[thm]{Corollary}

\theoremstyle{definition}
\newtheorem{defi}[thm]{Definition}
\newtheorem{rem}[thm]{Remark}

\numberwithin{equation}{subsection}

\author{Hideya Watanabe}
\address{(H. Watanabe) College of Science, Rikkyo University, 3-34-1, Nishi-Ikebukuro, Toshima-ku, Tokyo, 171-8501, Japan}
\email{watanabehideya@gmail.com}
\date{\today}

\title{Integrable modules over quantum symmetric pair coideal subalgebras}
\subjclass[2020]{Primary~17B10; Secondary~17B37}

\begin{document}
\maketitle

\begin{abstract}
  We introduce the notion of integrable modules over $\imath$quantum groups (a.k.a. quantum symmetric pair coideal subalgebras).
  After determining a presentation of such modules, we prove that each integrable module over a quantum group is integrable when restricted to an $\imath$quantum group.
  As an application, we show that the space of matrix coefficients of all simple integrable modules over an $\imath$quantum group of finite type with specific parameters coincides with Bao-Song's coordinate ring of the $\imath$quantum group.
\end{abstract}


\section{Introduction}\label{sec: intro}
\subsection{Representation theory}
The quantum groups (a.k.a. quantized enveloping algebras) were introduced by Drinfeld \cite{Dri85} and Jimbo \cite{Jim85} independently, and their representation theory has been studied for many years and been applied to a number of areas of mathematics and mathematical physics such as orthogonal polynomials, combinatorics, knot theory, and integrable systems.
One of the central objects in the representation theory of quantum groups is the integrable modules.
The notion of integrable modules originates in the representation theory of semisimple Lie algebras.
In this context, an integrable module is a module over a semisimple Lie algebra which can be integrated to a representation of a corresponding Lie group.
The integrability can be described as the local nilpotency of the Chevalley generators of the Lie algebra.
This description can be used to define the notion of integrable modules over quantum groups as it is.

The $\imath$quantum groups (a.k.a. quantum symmetric pair coideal subalgebras) of classical type were constructed by using solutions to reflection equation in some papers such as \cite{NoSu95}, \cite{Dij96}, \cite{Nou96} in order to perform noncommutative harmonic analysis on quantum symmetric spaces.
Letzter \cite{Let99} provided a unified construction of $\imath$quantum groups of all finite types (without relying on solutions to reflection equation), and generalized earlier results just mentioned. 
Her construction was further generalized to infinite types by Kolb \cite{Kol14}.

It has been pointed out that the representation theory of $\imath$quantum groups develops very slowly (\cite{Let19}, \cite{Wat21b}).
In fact, we do not still know any general theory to classify the finite-dimensional simple modules over $\imath$quantum groups of finite type, although such $\imath$quantum groups are quantum deformations of the universal enveloping algebra of a finite-dimensional complex reductive Lie algebra.
Several type-dependent classification can be found in e.g. \cite{GaKl91}, \cite{IoKl05}, \cite{Mol06}, \cite{ItTe10}, \cite{Wat20}, \cite{Wen20}, \cite{Wat21b}, \cite{KoSt24}.
Since the $\imath$quantum groups have close resemblance to the quantum groups, it seems to be essential for better understanding of their representation theory to formulate the notion of \emph{integrable modules} over $\imath$quantum groups.
This is the aim of the present paper.

\subsection{Integrable modules}
Let us recall the definition of integrable modules over quantum groups.
Let $\mathbf{U}$ be a quantum group, and $\{ E_i, F_i \mid i \in I \}$ its Chevalley generators.
Then, a weight $\mathbf{U}$-module $V$ is said to be integrable if for each weight vector $v \in V$, there exist tuples $(a_i)_{i \in I}$, $(b_i)_{i \in I}$ of nonnegative integers such that
\[
  E_i^{(a_i+1)} v = F_i^{(b_i+1)} v = 0,
\]
where the elements $E_i^{(a_i+1)}$, $F_i^{(b_i+1)}$ are divided powers.

Let us turn to the $\imath$quantum groups.
An $\imath$quantum group is a certain subalgebra of a quantum group $\mathbf{U}$, and has a distinguished family $\{ E_j, F_j, B_k \mid j \in I_\bullet,\ k \in I_\circ \}$ of elements, which play similar roles to Chevalley generators of $\mathbf{U}$ in some sense.
Here, $I_\bullet$ is a subset of $I$, and $I_\circ$ is its complement $I \setminus I_\bullet$.
An easy idea to define the integrable modules over the $\imath$quantum group is the local nilpotency with respect to $E_j$, $F_j$, and $B_k$.
However, this is not acceptable since $B_k$ may act semisimply on many modules which should be included in the class of integrable modules.
Bao and Wang, in their theory of $\imath$canonical bases \cite{BaWa18}, \cite{BaWa21}, introduced the notion of $\imath$divided powers, which are analogues of the divided powers of quantum groups.
Using the $\imath$divided powers, one may be able to resolve the semisimplicity issue of $B_k$ above.
However, testing integrability would be too hard since the $\imath$divided powers are very complicated in general.

To overcome this difficulty, let us recall an alternative definition of the integrable $\mathbf{U}$-modules.
Let $X^+$ denote the set of dominant weights.
For each $\lambda \in X^+$, let $V(\lambda)$ denote the integrable highest weight module of highest weight $\lambda$ with highest weight vector $v_\lambda$.
Also, let ${}^\omega V(\lambda)$ denote the integrable lowest weight module of lowest weight $-\lambda$ with lowest weight vector ${}^\omega v_\lambda$.
Then, a weight $\mathbf{U}$-module $V$ is integrable if and only if for each weight vector $v \in V$, there exist $\lambda, \mu \in X^+$ and a $\mathbf{U}$-module homomorphism ${}^\omega V(\lambda) \otimes V(\mu) \rightarrow V$ which sends ${}^\omega v_\lambda \otimes v_\mu$ to $v$.
Namely, integrable $\mathbf{U}$-modules are locally quotients of various ${}^\omega V(\lambda) \otimes V(\mu)$.

The $\mathbf{U}$-modules ${}^\omega V(\lambda) \otimes V(\mu)$ appear in Lusztig's construction of the canonical basis of modified form of $\mathbf{U}$ (\cite{Lus92}).
They are replaced by certain $\mathbf{U}^\imath$-modules $L^\imath(\lambda, \mu)$ in Bao-Wang's theory of $\imath$canonical bases \cite{BaWa18}, \cite{BaWa21}.
Hence, it is natural to define an integrable $\mathbf{U}^\imath$-module to be locally a quotient of $L^\imath(\lambda, \mu)$.
This definition turns out to work well in the present paper as summarized below.

\subsection{Results}
First, we determine a presentation of the $\mathbf{U}^\imath$-modules $L^\imath(\lambda, \mu)$ (Theorem \ref{thm: E+F+F_circ = E'+F'+B'}).
This enables one to test if a given $\mathbf{U}^\imath$-module is integrable in a systematic way.
In many cases, one only needs to check the local nilpotency with respect to $E_j$, $F_j$, $B_k$ just like the quantum group case.
In the remaining case, one needs to investigate local semisimplicity of some $B_k$.
As an application, we show that each integrable $\mathbf{U}$-module is integrable as a $\mathbf{U}^\imath$-module (Proposition \ref{prop: int Umod is int Uimod}).

Next, we concentrate on the $\imath$quantum groups of finite type.
We also assume that the parameters of the $\mathbf{U}^\imath$ is chosen in a way such that the $\mathbf{U}^\imath$ is invariant under a certain anti-algebra involution $\rho$ on $\mathbf{U}$, and that the $\imath$canonical basis of the modified form of $\mathbf{U}^\imath$ is stable (or, strongly compatible).
It turns out that each simple integrable $\mathbf{U}^\imath$-module is finite-dimensional and appears as a submodule of a finite-dimensional $\mathbf{U}$-module.

Let us consider the space of matrix coefficients of all simple integrable $\mathbf{U}^\imath$-modules.
As in the quantum group case (\cite[Section 7]{Kas93}), it admits a Peter-Weyl type decomposition.
Furthermore, by using the stability of $\imath$canonical bases and the characterization of integrability, we show that this space has a basis which is dual to the $\imath$canonical basis of the modified form of $\mathbf{U}^\imath$ (Theorem \ref{thm: dual icb spans Oi}).
This shows that the space of matrix coefficients coincides with Bao-Song's quantized coordinate ring \cite{BaSo22}, and hence, gives an intrinsic description of the latter.

\subsection{Organization}
The present paper is organized as follows.
In Section \ref{sec: qg}, we review basic definitions and results concerning the structure and representation theory of quantum groups.
Section \ref{sec: qsp} is devoted to formulating integrable modules over $\imath$quantum groups.
In Section \ref{sec: int ui mod}, we give an integrability criterion, and then prove that each integrable module over a quantum group is integrable when restricted to an $\imath$quantum group.
As an application of this result, we investigate the space of matrix coefficients of all simple integrable modules over an $\imath$quantum group of finite type in Section \ref{sec: qca}.

\subsection*{Acknowledgments}
The author thanks Paul Terwilliger for comments on the representation theory of $\imath$quantum groups, especially the $q$-Onsager algebra.
He is also grateful to Stefan Kolb for fruitful discussion on a sufficient condition for a weight $\mathbf{U}^\imath$-module to be integrable, which stimulates him to formulate Proposition \ref{prop: local nilp => integrable}.
He would like to express his gratitude to the anonymous referees for valuable comments and suggestions.
This work was supported by JSPS KAKENHI Grant Number JP24K16903.

\section{Quantum groups}\label{sec: qg}
In this section, we prepare necessary notation and fundamental results regarding the quantum groups.
We more or less follow Lusztig's textbook \cite{Lus93} except in Subsection \ref{subsect: levi and parabolic}, where we discuss Levi and parabolic subalgebras of a quantum group.

\subsection{Cartan and root data}
Throughout the present paper, we fix an indeterminate $q$, a Cartan datum $I = (I, \cdot)$ and a $Y$- and $X$-regular root datum $(Y, X, \langle , \rangle, \Pi^\vee, \Pi)$ of type $I$.
For each $i,j \in I$ and $n \in \mathbb{Z}_{\geq 0}$, set
\begin{align*}
  &d_i := \frac{i \cdot i}{2}, \\
  &a_{i,j} := \frac{2i \cdot j}{i \cdot i}, \\
  &q_i := q^{d_i}, \\
  &[n]_i := \frac{q_i^n - q_i^{-n}}{q_i - q_i^{-1}}, \\
  &[n]_i! := [n]_i [n-1]_i \cdots [1]_i.
\end{align*}
For each $i \in I$, let $h_i \in \Pi^\vee$ and $\alpha_i \in \Pi$ denote the corresponding simple coroot and simple root, respectively; in particular, we have
\[
  \langle h_i, \alpha_j \rangle = a_{i,j} \ \text{ for all } i,j \in I.
\]

Let $\mathrm{Br}$ and $W$ denote the braid group and Weyl group associated with $I$, respectively, and $s_i$ the simple reflection corresponding to $i \in I$.
They act on $\mathbb{Z}[I]$, $Y$, and $X$ by
\[
  s_i j = j - a_{i,j} i, \quad s_i h = h - \langle h, \alpha_i \rangle h_i, \quad s_i \lambda = \lambda - \langle h_i, \lambda \rangle \alpha_i
\]
for each $i,j \in I,\ h \in Y,\ \lambda \in X$.

Let $Q := \sum_{i \in I} \mathbb{Z} \alpha_i \subseteq X$ denote the root lattice, and $Q^\pm := \pm \sum_{i \in I} \mathbb{Z}_{\geq 0} \alpha_i$ the positive and negative cones.

Let $\leq$ denote the dominance order on $X$, the partial ordering defined by declaring $\lambda \leq \mu$ to mean $\mu - \lambda \in Q^+$.

Let $X^+$ denote the set of dominant weights:
\[
  X^+ := \{ \lambda \in X \mid \langle h_i, \lambda \rangle \geq 0 \ \text{ for all } i \in I \}.
\]

\subsection{Quantum groups}
Let $\mathbf{U}$ denote the quantum group associated with the root datum $(Y, X, \langle , \rangle, \Pi^\vee, \Pi)$.
Namely, $\mathbf{U}$ is the unital associative $\mathbb{Q}(q)$-algebra with generators
\[
  \{ E_i,\ F_i,\ K_h \mid i \in I,\ h \in Y \}
\]
subject to the following relations:
for each $h,h_1,h_2 \in Y$ and $i,j \in I$,
\begin{align*}
  &K_0 = 1, \\
  &K_{h_1} K_{h_2} = K_{h_1 + h_2}, \\
  &K_h E_i = q^{\langle h, \alpha_i \rangle} E_i K_h, \\
  &K_h F_i = q^{\langle h, -\alpha_i \rangle} F_i K_h, \\
  &E_i F_j - F_j E_i = \delta_{i,j} \frac{K_i - K_i^{-1}}{q_i - q_i^{-1}}, \\
  &\sum_{r+s = 1-a_{i,j}} (-1)^s E_i^{(r)} E_j E_i^{(s)} = 0 \ \text{ if } i \ne j, \\
  &\sum_{r+s = 1-a_{i,j}} (-1)^s F_i^{(r)} F_j F_i^{(s)} = 0 \ \text{ if } i \ne j.
\end{align*}
where
\[
  K_i := K_{d_i h_i}, \quad E_i^{(n)} := \frac{1}{[n]_i!} E_i^n, \quad F_i^{(n)} := \frac{1}{[n]_i!} F_i^n.
\]

The braid group $\mathrm{Br}$ acts on $\mathbf{U}$ by $s_i \mapsto T''_{i,1}$ (\cite[\S 37.1]{Lus93}).
For each $w \in W$ with reduced expression $w = s_{i_1} \cdots s_{i_l}$, set
\[
  T_w := T''_{i_1,1} \cdots T''_{i_l,1}.
\]

Let $\mathbf{U}^-$, $\mathbf{U}^0$, and $\mathbf{U}^+$ denote the negative, Cartan, and positive parts of $\mathbf{U}$, respectively.
We have a triangular decomposition:
\begin{align}\label{eq: triangular decomposition of U}
  \mathbf{U} = \mathbf{U}^- \mathbf{U}^0 \mathbf{U}^+ \simeq \mathbf{U}^- \otimes \mathbf{U}^0 \otimes \mathbf{U}^+.
\end{align}

Let $\mathbf{B}(\pm \infty)$ denote the canonical bases of $\mathbf{U}^\mp$, respectively.
They are graded by $Q^\mp$:
\[
  \mathbf{B}(\pm \infty) = \bigsqcup_{\lambda \in Q^\mp} \mathbf{B}(\pm \infty)_\lambda,
\]
where
\[
  \mathbf{B}(\pm \infty)_\lambda := \mathbf{B}(\pm \infty) \cap \mathbf{U}^\mp_\lambda, \quad \mathbf{U}^\mp_\lambda := \{ u \in \mathbf{U}^\mp \mid K_h u K_{-h} = q^{\langle h, \lambda \rangle} u \ \text{ for all } h \in Y \}.
\]
For each $b \in \mathbf{B}(\pm \infty)_\lambda$, we set $\operatorname{wt}(b) := \lambda$.

Let $\omega$ denote the Chevalley involution on $\mathbf{U}$, that is, the algebra automorphism such that
\[
  \omega(E_i) = F_i, \ \omega(K_h) = K_{-h} \ \text{ for all } i \in I,\ h \in Y.
\]
Given a $\mathbf{U}$-module $V$, let ${}^\omega V = \{ {}^\omega v \mid v \in V \}$ denote the $\mathbf{U}$-module $V$ twisted by $\omega$:
\[
  u \cdot {}^\omega v = {}^\omega(\omega(u) v) \ \text{ for each } u \in \mathbf{U},\ v \in V.
\]

Let $\dot{\mathbf{U}}$ denote the modified form of $\mathbf{U}$ with idempotents $\{ 1_\lambda \mid \lambda \in X \}$, and $\dot{\mathbf{B}}$ the canonical basis of $\dot{\mathbf{U}}$.

\subsection{Weight and integrable modules}
A $\mathbf{U}$-module $V$ is said to be a weight module if it admits a linear space decomposition
\[
  V = \bigoplus_{\lambda \in X} V_\lambda
\]
such that
\[
  V_\lambda = \{ v \in V \mid K_h v = q^{\langle h, \lambda \rangle} v \ \text{ for all } h \in Y \}.
\]
For each $v \in V_\lambda \setminus \{ 0 \}$, we set $\operatorname{wt}(v) := \lambda$.

For each $\lambda \in X$, let $M(\lambda)$ denote the Verma module of highest weight $\lambda$ with highest weight vector $m_\lambda$.

For each $\lambda,\mu \in X$, set
\[
  M(\lambda, \mu) := {}^\omega M(\lambda) \otimes M(\mu)
\]
and
\[
  m_{\lambda, \mu} := {}^\omega m_\lambda \otimes m_\mu \in M(\lambda, \mu).
\]
The linear map
\[
  \dot{\mathbf{U}} 1_{-\lambda + \mu} \rightarrow M(\lambda, \mu);\ u \mapsto u m_{\lambda, \mu}
\]
is a $\mathbf{U}$-module isomorphism (\cite[23.3.1 (c)]{Lus93}).
The $\mathbf{U}$-module $M(\lambda, \mu)$ (or rather $\dot{\mathbf{U}} 1_{-\lambda+\mu}$) is a universal weight module of weight $-\lambda + \mu$ in the following sense.
For each weight $\mathbf{U}$-module $V$ and a weight vector $v \in V_{-\lambda+\mu}$, there exists a unique $\mathbf{U}$-module homomorphism $M(\lambda, \mu) \rightarrow V$ which sends $m_{\lambda, \mu}$ to $v$.

A weight $\mathbf{U}$-module $V$ is said to be integrable if for each $\lambda \in X$ and $v \in V_\lambda$, there exist $(a_i)_{i \in I}, (b_i)_{i \in I} \in \mathbb{Z}_{\geq 0}^I$ such that
\[
  E_i^{(a_i+1)} v = F_i^{(b_i+1)} v = 0 \ \text{ for all } i \in I.
\]

For each $\lambda \in X^+$, let $V(\lambda)$ denote the integrable highest weight module of highest weight $\lambda$.
It is the maximal integrable quotient of the Verma module $M(\lambda)$:
\[
  V(\lambda) := M(\lambda)/\sum_{i \in I} \mathbf{U} F_i^{(\langle h_i, \lambda \rangle + 1)} m_\lambda.
\]
Twisting it by the Chevalley involution $\omega$, we obtain
\begin{align}\label{eq: def rel of wV(lm)}
  {}^\omega V(\lambda) \simeq {}^\omega M(\lambda)/\sum_{i \in I} \mathbf{U} E_i^{(\langle h_i, \lambda \rangle + 1)} \cdot {}^\omega m_\lambda.
\end{align}
Let $v_\lambda \in V(\lambda)$ denote the image of the highest weight vector $m_\lambda \in M(\lambda)$.
Let $\mathbf{B}(\lambda)$ denote the canonical basis of $V(\lambda)$.
The linear map
\[
  \mathbf{U}^+ \rightarrow {}^\omega V(\lambda);\ u \mapsto u \cdot {}^\omega v_\lambda
\]
gives rise to a bijection
\[
  \mathbf{B}(-\infty)[-\lambda] := \{ b \in \mathbf{B}(-\infty) \mid b \cdot {}^\omega v_\lambda \neq 0 \} \rightarrow {}^\omega \mathbf{B}(\lambda) := \{ {}^\omega b \mid b \in \mathbf{B}(\lambda) \}.
\]
Combining this fact with the presentation \eqref{eq: def rel of wV(lm)}, we obtain
\begin{align}\label{eq: def rel of wV(lm) in cb}
  \sum_{b \in \mathbf{B}(-\infty) \setminus \mathbf{B}(-\infty)[-\lambda]} \mathbf{U} b \cdot {}^\omega m_\lambda = \sum_{i \in I} \mathbf{U} E_i^{(\langle h_i, \lambda \rangle + 1)} \cdot {}^\omega m_\lambda.
\end{align}

For each $\lambda, \mu \in X^+$, set
\[
  V(\lambda, \mu) := {}^\omega V(\lambda) \otimes V(\mu)
\]
and
\[
  v_{\lambda, \mu} := {}^\omega v_\lambda \otimes v_\mu \in V(\lambda, \mu).
\]
By \cite[Proposition 23.3.6]{Lus93}, we have
\begin{align}\label{eq: presentation V(lm,mu)}
  V(\lambda,\mu) \simeq M(\lambda,\mu)/(\sum_{i \in I} \mathbf{U} E_i^{(\langle h_i, \lambda \rangle + 1)} m_{\lambda,\mu} + \sum_{i \in I} \mathbf{U} F_i^{(\langle h_i, \mu \rangle + 1)} m_{\lambda,\mu}).
\end{align}

The $\mathbf{U}$-modules $V(\lambda, \mu)$ are universal integrable $\mathbf{U}$-modules in the following sense.
Let $V$ be an integrable $\mathbf{U}$-module and $v \in V$ a weight vector.
Then, there exist $\lambda, \mu \in X^+$ and a $\mathbf{U}$-module homomorphism $V(\lambda, \mu) \rightarrow V$ which sends $v_{\lambda, \mu}$ to $v$ (\cite[Proposition 23.3.10]{Lus93}).

\subsection{Levi and parabolic subalgebras}\label{subsect: levi and parabolic}
In this subsection, we fix a subset $J \subseteq I$, and set $K := I \setminus J$.
We often regard $J$ as a Cartan datum and
\[
  (Y, X, \langle , \rangle, \{ h_j \mid j \in J \}, \{ \alpha_j \mid j \in J \})
\]
as a root datum of type $J$.

Let $\mathbf{L} = \mathbf{L}_J$ denote the Levi subalgebra of the quantum group $\mathbf{U}$ associated with $J$, that is, the subalgebra generated by
\[
  \{ E_j,\ F_j,\ K_h \mid j \in J,\ h \in Y \}.
\]
The algebra $\mathbf{L}$ itself is the quantum group associated with the root datum of type $J$ above.
Let $\mathbf{L}^-$ and $\mathbf{L}^+$ denote the negative and the positive parts of $\mathbf{L}$, respectively.
Noting that the Cartan part of $\mathbf{L}$ is $\mathbf{U}^0$, we have a triangular decomposition
\begin{align}\label{eq: triangular decomposition of L}
  \mathbf{L} = \mathbf{L}^- \mathbf{U}^0 \mathbf{L}^+ \simeq \mathbf{L}^- \otimes \mathbf{U}^0 \otimes \mathbf{L}^+.
\end{align}

For each $\lambda \in X$, let $M_\mathbf{L}(\lambda)$ denote the Verma module over $\mathbf{L}$ of highest weight $\lambda$ with highest weight vector $m_{\mathbf{L}, \lambda}$.
Similarly, for each $\lambda, \mu \in X$, set $M_\mathbf{L}(\lambda, \mu) := {}^\omega M_\mathbf{L}(\lambda) \otimes M_\mathbf{L}(\mu)$ and $m_{\mathbf{L};\lambda,\mu} := {}^\omega m_{\mathbf{L},\lambda} \otimes m_{\mathbf{L},\mu} \in M_\mathbf{L}(\lambda, \mu)$.
Then, we have $M_\mathbf{L}(\lambda, \mu) \simeq \dot{\mathbf{L}} 1_{-\lambda + \mu}$ as $\mathbf{L}$-modules, where $\dot{\mathbf{L}}$ denotes the modified form of $\mathbf{L}$.

Let $\mathbf{P}^+ = \mathbf{P}^+_J$ and $\mathbf{P}^- = \mathbf{P}^-_J$ denote the parabolic and the opposite parabolic subalgebra of $\mathbf{U}$ associated with the subset $J$, respectively.
Namely, $\mathbf{P}^+$ is the subalgebra of $\mathbf{U}$ generated by $\mathbf{L}$ and $\{ E_k \mid k \in K \}$, and $\mathbf{P}^-$ generated by $\mathbf{L}$ and $\{ F_k \mid k \in K \}$.

We have triangular decompositions:
\begin{align}
  &\mathbf{P}^+ = \mathbf{L}^- \mathbf{U}^0 \mathbf{U}^+ \simeq \mathbf{L}^- \otimes \mathbf{U}^0 \otimes \mathbf{U}^+; \label{eq: triangular decomposition of P+}\\
  &\mathbf{P}^- = \mathbf{U}^- \mathbf{U}^0 \mathbf{L}^+ \simeq \mathbf{U}^- \otimes \mathbf{U}^0 \otimes \mathbf{L}^+; \label{eq: triangular decomposition of P-}
\end{align}

Although the algebras $\mathbf{P}^\pm$ are not quantum groups, we can construct their modified forms $\dot{\mathbf{P}}^\pm = \bigoplus_{\lambda \in X} \dot{\mathbf{P}}^\pm 1_\lambda$ in a canonical way.
Just like the modified forms of quantum groups, the $\dot{\mathbf{P}}^\pm$ have   natural $\mathbf{P}^\pm$-bimodule structure.
Furthermore, the $\mathbf{L}$-bimodule structure on $\mathbf{P}^\pm$ gives rise to an $\dot{\mathbf{L}}$-bimodule structure on $\dot{\mathbf{P}}^\pm$.

Let $\mathbf{R}^\pm = \mathbf{R}^\pm_J$ denote the nilradical part of $\mathbf{P}^\pm$.
That is, $\mathbf{R}^+$ is the two-sided ideal of $\mathbf{U}^+$ generated by $\{ E_k \mid k \in K \}$, and $\mathbf{R}^-$ generated by $\{ F_k \mid k \in K \}$.
Then, we have
\begin{align}\label{eq: levi decomposition of Upm}
  \mathbf{U}^\pm = \mathbf{L}^\pm \oplus \mathbf{R}^\pm.
\end{align}

\begin{prop}\label{prop: levi decomposition of P+}
  As a $\mathbb{Q}(q)$-linear space, we have
    \begin{align}\label{eq: levi decomp of P+}
      \mathbf{P}^+ = \mathbf{L} \oplus \mathbf{L}^- \mathbf{U}^0 \mathbf{R}^+.
    \end{align}
  Moreover, the subspace $\mathbf{L}^- \mathbf{U}^0 \mathbf{R}^+$ coincides with the two-sided ideal of $\mathbf{P}^+$ generated by $\{ E_k \mid k \in K \}$.
\end{prop}

\begin{proof}
  By the decompositions \eqref{eq: triangular decomposition of P+}, \eqref{eq: levi decomposition of Upm}, and \eqref{eq: triangular decomposition of L} we have
  \[
    \mathbf{P}^+ = \mathbf{L}^- \mathbf{U}^0 (\mathbf{L}^+ \oplus \mathbf{R}^+) = \mathbf{L} \oplus \mathbf{L}^- \mathbf{U}^0 \mathbf{R}^+.
  \]
  This proves the first assertion.

  To prove the second assertion, let us compute as
  \begin{align}\label{eq: L^-U^0R^+ = sum_k P^+E_kU^+}
    \mathbf{L}^- \mathbf{U}^0 \mathbf{R}^+ = \sum_{k \in K} \mathbf{L}^- \mathbf{U}^0 \mathbf{U}^+ E_k \mathbf{U}^+ = \sum_{k \in K} \mathbf{P}^+ E_k \mathbf{U}^+.
  \end{align}
  For each $k \in K$, we have
  \begin{align}\label{eq: P^+E_kP^+ = P^+E_kU^+}
    \mathbf{P}^+ E_k \mathbf{P}^+
    = \mathbf{P}^+ E_k \mathbf{L}^- \mathbf{U}^0 \mathbf{U}^+
    = \mathbf{P}^+ \mathbf{L}^- \mathbf{U}^0 E_k \mathbf{U}^+
    = \mathbf{P}^+ E_k \mathbf{U}^+.
  \end{align}
  Combining equations \eqref{eq: L^-U^0R^+ = sum_k P^+E_kU^+} and \eqref{eq: P^+E_kP^+ = P^+E_kU^+}, we obtain
  \[
    \mathbf{L}^- \mathbf{U}^0 \mathbf{R}^+ = \sum_{k \in K} \mathbf{P}^+ E_k \mathbf{U}^+ = \sum_{k \in K} \mathbf{P}^+ E_k \mathbf{P}^+,
  \]
  as desired.
  Thus, we complete the proof.
\end{proof}

\begin{lem}\label{lem: P- simeq U/UR+}
  Let $\lambda \in X$.
  We denote by $1^\mathbf{U}_\lambda$ and $1^{\mathbf{P}^-}_\lambda$ the idempotents in $\dot{\mathbf{U}}$ and $\dot{\mathbf{P}}^-$ corresponding to $\lambda$, respectively.
  Then, as $\mathbb{Q}(q)$-linear spaces, we have
  \begin{align}\label{eq: P^- simeq U/UR^+}
    \mathbf{P}^- 1^{\mathbf{P}^-}_\lambda \simeq \mathbf{U} 1^\mathbf{U}_\lambda/\mathbf{U} \mathbf{R}^+ 1^\mathbf{U}_\lambda.
  \end{align}
\end{lem}

\begin{proof}
  By the decompositions \eqref{eq: triangular decomposition of U} and \eqref{eq: levi decomposition of Upm}, we have
  \[
    \mathbf{U} 1^\mathbf{U}_\lambda
    = \mathbf{U}^- \mathbf{U}^0 \mathbf{U}^+ 1^\mathbf{U}_\lambda
    = \mathbf{U}^- \mathbf{U}^0 (\mathbf{L}^+ \oplus \mathbf{R}^+) 1^\mathbf{U}_\lambda
    = (\mathbf{P}^- \oplus \mathbf{U} \mathbf{R}^+) 1^\mathbf{U}_\lambda.
  \]
  Hence, we obtain a linear isomorphism
  \begin{align}\label{eq: P1U simeq U1/UR+1}
    \mathbf{P}^- 1^\mathbf{U}_\lambda \simeq \mathbf{U} 1^\mathbf{U}_\lambda/\mathbf{U} \mathbf{R}^+ 1^\mathbf{U}_\lambda.
  \end{align}
  Since the triangular decomposition \eqref{eq: triangular decomposition of P-} of $\mathbf{P}^-$ is consistent with the one \eqref{eq: triangular decomposition of U} of $\mathbf{U}$, both the linear spaces $\mathbf{P}^- 1^\mathbf{U}_\lambda$ and $\mathbf{P}^- 1^{\mathbf{P}^-}_\lambda$ are isomorphic to $\mathbf{U}^- \otimes \mathbf{L}^+$ in canonical ways.
  Thus, we complete the proof.
\end{proof}

\begin{rem}
  The linear isomorphism \eqref{eq: P1U simeq U1/UR+1} has appeared in \cite[\S 3.5]{BaWa21}.
\end{rem}

\begin{prop}
  Let $V$ be a weight $\mathbf{L}$-module.
  Then, as a $\mathbf{U}$-module, we have
  \[
    \operatorname{Ind}_{\mathbf{P}^+}^\mathbf{U} V := \mathbf{U} \otimes_{\mathbf{P}^+} V = \dot{\mathbf{P}}^- \otimes_{\dot{\mathbf{L}}} V;
  \]
  here, we regard $V$ as a $\mathbf{P}^+$-module via the projection $\mathbf{P}^+ \rightarrow \mathbf{L}$ with respect to the decomposition \eqref{eq: levi decomp of P+}, and $\dot{\mathbf{P}}^- (= \bigoplus_{\lambda \in X} \mathbf{P}^- 1_\lambda)$ as a $\mathbf{U}$-module via the linear isomorphism \eqref{eq: P^- simeq U/UR^+}.
\end{prop}

\begin{proof}
  Let
  \[
    \phi \colon \operatorname{Ind}_{\mathbf{P}^+}^\mathbf{U} V \rightarrow \dot{\mathbf{P}}^- \otimes_{\dot{\mathbf{L}}} V \text{ and } \psi \colon \dot{\mathbf{P}}^- \otimes_{\dot{\mathbf{L}}} V \rightarrow \operatorname{Ind}_{\mathbf{P}^+}^\mathbf{U} V
  \]
  denote the linear maps given as follows.
  Let $\lambda \in X$, $v \in V_\lambda$, $\mu \in Q^+$, $u^- \in \mathbf{U}^-$, $u^0 \in \mathbf{U}^0$, $u^+ \in \mathbf{U}^+_\mu$, and $p \in \mathbf{P}^-$.
  Then, we set
  \[
    \phi(u^-u^0u^+ \otimes v) := u^- 1_{\lambda+\mu} \otimes u^0u^+v \text{ and } \psi(p 1_\lambda \otimes v) := p \otimes v.
  \]
  That these maps are well-defined can be straightforwardly verified.
  Clearly, we have $\phi \circ \psi = \mathrm{id}$, and see that $\psi$ is a $\mathbf{U}$-module homomorphism.
  Thus, we complete the proof.
\end{proof}

\begin{cor}\label{cor: Ind ML simeq Pdot1}
  Let $\lambda, \mu \in X$.
  Then, as $\mathbf{U}$-modules, we have
  \[
    \operatorname{Ind}_{\mathbf{P}^+}^\mathbf{U} M_\mathbf{L}(\lambda, \mu)
    \simeq \operatorname{Ind}_{\mathbf{P}^+}^\mathbf{U} \dot{\mathbf{L}} 1_{-\lambda + \mu}
    \simeq \dot{\mathbf{P}}^- \otimes_{\dot{\mathbf{L}}} \dot{\mathbf{L}} 1_{-\lambda + \mu}
    \simeq \dot{\mathbf{P}}^- 1_{-\lambda+\mu}.
  \]
\end{cor}

\section{Quantum symmetric pairs}\label{sec: qsp}
After formulating basic notation regarding quantum symmetric pairs following mainly \cite{BaWa21}, we will introduce the notion of \emph{integrable modules} over an $\imath$quantum group.

\subsection{Quantum symmetric pair}
Let $(I, I_\bullet, \tau)$ be a generalized Satake diagram (\cite[Definition 1]{ReVl20}).
From now on, we assume that the root datum $(Y, X, \langle , \rangle, \Pi^\vee, \Pi)$ is a Satake datum of type $(I, I_\bullet, \tau)$ (\cite[Definition 3.1.3]{Wat23e}).
In particular, $I_\bullet$ is a subdatum of $I$ which is supposed to be of finite type, $\tau$ is an involutive automorphism on the Cartan datum $I$, and $Y$ and $X$ are supposed to be equipped with involutive automorphisms $\tau$ which are compatible with the one on $I$.
Let $w_\bullet \in W$ denote the longest element of the Weyl group $W_\bullet \subseteq W$ associated with $I_\bullet$.
Set
\begin{align*}
  &\theta := -w_\bullet \tau \in \operatorname{Aut}(\mathbb{Z}[I]), \operatorname{Aut}(Y), \operatorname{Aut}(X); \\
  &Y^\imath := \{ h \in Y \mid \theta(h) = h \}; \\
  &X^\imath := X/\{ \lambda - \theta(\lambda) \mid \lambda \in X \}; \\
  &I_\circ := I \setminus I_\bullet.
\end{align*}
Also, let
\[
  \overline{\cdot} \colon X \rightarrow X^\imath
\]
denote the quotient map, and
\[
  \langle , \rangle \colon Y^\imath \times X^\imath \rightarrow \mathbb{Z}
\]
the bilinear pairing induced from the one on $Y \times X$.

For later use, let us recall one of the axioms for generalized Satake diagrams from (\cite[(2.19)]{ReVl20}):
\begin{align}
  \langle h_i, \theta(\alpha_i) \rangle \neq -1 \ \text{ for all } i \in I.\label{gSat4}
\end{align}

\begin{lem}\label{lem: values in axioms for gSat}
  Let $k \in I_\circ$.
  If $\tau(k) = k \neq w_\bullet k$, then we have
  \[
    \langle h_k, w_\bullet \alpha_k - \alpha_k \rangle \leq -2.
  \]
\end{lem}

\begin{proof}
  Since $w_\bullet \alpha_k - \alpha_k \in \sum_{j \in I_\bullet} \mathbb{Z}_{\geq 0} \alpha_j$, we have
  \[
    \langle h_k, w_\bullet \alpha_k - \alpha_k \rangle \in \mathbb{Z}_{\leq 0}.
  \]
  Our assumption on $k$ that $w_\bullet k \neq k$ implies that this value cannot be $0$.

  On the other hand, taking into account the assumption that $\tau(k) = k$, we compute as follows:
  \[
    \langle h_k, w_\bullet \alpha_k - \alpha_k \rangle = \langle h_k, -\theta(\alpha_k) - \alpha_k \rangle = -2 - \langle h_k, \theta(\alpha_k) \rangle.
  \]
  By the axiom \eqref{gSat4}, the rightmost-hand side cannot be $-1$.
  Hence, the assertion follows.
\end{proof}

Let $\mathbf{U}^\imath$ denote the $\imath$quantum group associated with the Satake datum $(Y, X, \langle , \rangle, \Pi^\vee, \Pi)$ and parameters ${\boldsymbol \varsigma} = (\varsigma_k)_{k \in I_\circ} \in (\mathbb{Q}(q)^\times)^{I_\circ}$, ${\boldsymbol \kappa} = (\kappa_i)_{k \in I_\circ} \in \mathbb{Q}(q)^{I_\circ}$.
It is the subalgebra of the quantum group $\mathbf{U}$ generated by
\[
  \{ E_j,\ F_j,\ B_k,\ K_h \mid j \in I_\bullet,\ k \in I_\circ,\ h \in Y^\imath \},
\]
where
\[
  B_k := F_k + \varsigma_k T_{w_\bullet}(E_{\tau(k)}) K_k^{-1} + \kappa_k K_k^{-1}
\]
for each $k \in I_\circ$.
The parameters are supposed to satisfy the following for all $k \in I_\circ$:
\begin{itemize}
  \item $\varsigma_k = \varsigma_{\tau(k)}$ if $k \cdot \theta(k) = 0$.
  \item $\kappa_k = 0$ unless $\tau(k) = k$, $a_{k,j} = 0$ for all $j \in I_\bullet$, and $a_{k,k'} \in 2\mathbb{Z}$ for all $k' \in I_\circ$ such that $\tau(k') = k'$ and $a_{k',j} = 0$ for all $j \in I_\bullet$.
\end{itemize}
They are as general as \cite[Definition 5.6]{Kol14}.

Let $\dot{\mathbf{U}}^\imath$ denote the modified form of $\mathbf{U}^\imath$ with idempotents $\{ 1_\zeta \mid \zeta \in X^\imath \}$.

Let $\mathbf{L}$, $\mathbf{P}^\pm$, and $\mathbf{R}^\pm$ denote the Levi, parabolic subalgebras, and the nilradical parts associated with $I_\bullet$, respectively.

\begin{prop}[{{\it cf}.\ \cite[Lemma 3.22]{BaWa18}}]\label{prop: Uidot simeq Pdot-}
  Let $\lambda \in X$ and $\zeta := \bar{\lambda} \in X^\imath$.
  As $\mathbf{U}^\imath$-modules, we have
  \[
    \dot{\mathbf{U}}^\imath 1_\zeta \simeq \mathbf{U}^\imath 1_\lambda \simeq \dot{\mathbf{P}}^- 1_\lambda.
  \]
\end{prop}

\subsection{Weight modules}
A $\mathbf{U}^\imath$-module $V$ is said to be a weight module (\cite[Definition 3.3.2]{Wat23}) if it admits a linear space decomposition
\[
  V = \bigoplus_{\zeta \in X^\imath} V_\zeta
\]
such that
\begin{itemize}
  \item $V_\zeta = \{ v \in V \mid K_h v = q^{\langle h, \zeta \rangle} v \ \text{ for all } h \in Y^\imath \}$;
  \item $E_j V_\zeta \subseteq V_{\zeta + \overline{\alpha_j}}$, $F_j V_\zeta \subseteq V_{\zeta - \overline{\alpha_j}}$ for all $j \in I_\bullet$;
  \item $B_k V_\zeta \subseteq V_{\zeta - \overline{\alpha_k}}$ for all $k \in I_\circ$;
\end{itemize}

A weight $\mathbf{U}$-module $V = \bigoplus_{\lambda \in X} V_\lambda$ has a canonical weight $\mathbf{U}^\imath$-module structure:
\begin{align}\label{eq: wt Umod as wt Uimod}
  V = \bigoplus_{\zeta \in X^\imath} V_\zeta, \ V_\zeta := \bigoplus_{\substack{\lambda \in X \\ \bar{\lambda} = \zeta}} V_\lambda.
\end{align}

For each $\lambda, \mu \in X$, set
\[
  M^\imath(\lambda, \mu) := \operatorname{Ind}_{\mathbf{P}^+}^\mathbf{U} M_\mathbf{L}(-w_\bullet \lambda, \mu)
\]
and
\[
  m^\imath_{\lambda, \mu} := 1 \otimes m_{\mathbf{L}; -w_\bullet \lambda, \mu} \in M^\imath(\lambda, \mu).
\]
Then, by Corollary \ref{cor: Ind ML simeq Pdot1}, we have
\begin{align}\label{eq: Mi simeq P^-}
  M^\imath(\lambda, \mu) \simeq \operatorname{Ind}_{\mathbf{P}^+}^\mathbf{U} \dot{\mathbf{L}} 1_{w_\bullet \lambda + \mu} \simeq \dot{\mathbf{P}}^- 1_{w_\bullet \lambda + \mu}
\end{align}
as $\mathbf{U}$-modules.
The $\mathbf{U}^\imath$-module $M^\imath(\lambda, \mu)$ is a universal weight module of weight $\overline{w_\bullet \lambda + \mu}$ in the following sense.

\begin{prop}
  Let $\lambda, \mu \in X$ and set $\zeta := \overline{w_\bullet \lambda + \mu} \in X^\imath$.
  Let $V$ be a weight $\mathbf{U}^\imath$-module and $v \in V_\zeta$.
  Then, there exists a unique $\mathbf{U}^\imath$-module homomorphism $M^\imath(\lambda, \mu) \rightarrow V$ which sends $m^\imath_{\lambda, \mu}$ to $v$.
\end{prop}

\begin{proof}
  By the isomorphism \eqref{eq: Mi simeq P^-} and Proposition \ref{prop: Uidot simeq Pdot-}, we have
  \[
    M^\imath(\lambda,\mu) \simeq \dot{\mathbf{U}}^\imath 1_\zeta
  \]
  as $\mathbf{U}^\imath$-modules.
  Then, the assertion can be found in \cite[after Definition 3.3.5]{Wat23}.
\end{proof}

\subsection{Integrable modules}
For each $\lambda, \mu \in X^+$, set $V^\imath(\lambda, \mu)$ to be the quotient $\mathbf{U}$-module of $M^\imath(\lambda, \mu)$ factored by the submodule generated by
\begin{align}\label{eq: def rel of Vi as quot of Mi}
  \{ E_j^{(\langle h_j, -w_\bullet \lambda \rangle + 1)} m^\imath_{\lambda, \mu},\ F_j^{(\langle h_j, \mu \rangle + 1)} m^\imath_{\lambda, \mu},\ F_k^{(\langle h_k, w_\bullet\lambda + \mu \rangle + 1)} m^\imath_{\lambda,\mu} \mid j \in I_\bullet,\ k \in I_\circ \}.
\end{align}
Let $v^\imath_{\lambda,\mu} \in V^\imath(\lambda,\mu)$ denote the image of $m^\imath_{\lambda,\mu} \in M^\imath(\lambda,\mu)$ under the quotient map.

\begin{lem}\label{lem: def rel of Vi as Umod}
  Let $\lambda, \mu \in X^+$.
  Then, the $\mathbf{U}$-module $V^\imath(\lambda, \mu)$ is isomorphic to the quotient module of $\dot{\mathbf{U}} 1_{w_\bullet \lambda + \mu}$ factored by the submodule generated by $\mathbf{R}^+ 1_{w_\bullet \lambda + \mu}$ and the set \eqref{eq: def rel of Vi as quot of Mi}
\end{lem}

\begin{proof}
  The assertion follows from the definition of $V^\imath(\lambda, \mu)$ and the isomorphisms \eqref{eq: Mi simeq P^-} and \eqref{eq: P^- simeq U/UR^+}.
\end{proof}

\begin{lem}\label{lem: Umod str is determined by Lmod str}
  Let $\lambda, \mu \in X^+$.
  Then, there exist a finite number of dominant weights $\nu_1,\dots,\nu_r \in X^+$ such that
  \[
    \mathbf{L} v^\imath_{\lambda, \mu} \simeq \bigoplus_{s=1}^r V_\mathbf{L}(\nu_s)
  \]
  as $\mathbf{L}$-modules and
  \[
    V^\imath(\lambda, \mu) \simeq \bigoplus_{s=1}^r V(\nu_s)
  \]
  as $\mathbf{U}$-modules, where $V_\mathbf{L}(\nu)$ denotes the integrable highest weight $\mathbf{L}$-module of highest weight $\nu \in X_\mathbf{L}^+$.
\end{lem}

\begin{proof}
  By the definition of $V^\imath(\lambda, \mu)$ and the presentation \eqref{eq: presentation V(lm,mu)} of $V_\mathbf{L}(-w_\bullet\lambda, \mu)$, there exists an $\mathbf{L}$-module homomorphism
  \[
    V_\mathbf{L}(-w_\bullet \lambda, \mu) \rightarrow \mathbf{L} v^\imath_{\lambda, \mu}
  \]
  which sends $v_{\mathbf{L};-w_\bullet\lambda, \mu}$ to $v^\imath_{\lambda, \mu}$.
  This homomorphism is clearly surjective.
  Since the domain is finite-dimensional, so is the codomain.
  This implies that there exist $\nu_1,\dots,\nu_r \in X_\mathbf{L}^+$ such that
  \[
    \mathbf{L} v^\imath_{\lambda, \mu} \simeq \bigoplus_{s=1}^r V_\mathbf{L}(\nu_s).
  \]

  On the other hand, the $\mathbf{U}$-module $V^\imath(\lambda, \mu)$ is integrable, and
  \begin{align}\label{eq: U^+vi = L^+vi}
    \mathbf{U}^+ v^\imath_{\lambda, \mu} = \mathbf{L}^+ v^\imath_{\lambda, \mu}
  \end{align}
  Since the right-hand side is finite-dimensional, the weights of $V^\imath(\lambda, \mu)$ are bounded above.
  Therefore, the $\mathbf{U}$-module is a direct sum of submodules isomorphic to integrable highest weight modules.
  By equation \eqref{eq: U^+vi = L^+vi}, we see that the highest weight vectors in the $\mathbf{U}$-module $V^\imath(\lambda, \mu)$ are exactly those in the $\mathbf{L}$-module $\mathbf{L} v^\imath_{\lambda, \mu}$.
  This implies that
  \[
    \nu_1,\dots,\nu_r \in X^+
  \]
  and
  \[
    V^\imath(\lambda, \mu) \simeq \bigoplus_{s=1}^r V(\nu_s).
  \]
  Thus, we complete the proof.
\end{proof}

\begin{prop}\label{prop: Vi simeq Li}
  Let $L^\imath(\lambda,\mu)$ denote the $\mathbf{U}$-submodule of $V(\lambda) \otimes V(\mu)$ generated by $v_{w_\bullet \lambda} \otimes v_\mu$, where $v_{w_\bullet \lambda}$ denotes the unique element in $\mathbf{B}(\lambda)$ of weight $w_\bullet \lambda$.
  Then, we have
  \[
    V^\imath(\lambda,\mu) \simeq L^\imath(\lambda,\mu)
  \]
  as $\mathbf{U}$-modules.
\end{prop}

\begin{proof}
  It is easily verified that $\mathbf{R}^+ (v_{w_\bullet \lambda} \otimes v_\mu) = 0$ and
  \[
    E_j^{(\langle h_j, -w_\bullet \lambda \rangle + 1)} (v_{w_\bullet \lambda} \otimes v_\mu) = F_j^{(\langle h_j, \mu \rangle + 1)} (v_{w_\bullet \lambda} \otimes v_\mu) = F_k^{(\langle h_k, w_\bullet \lambda + \mu \rangle + 1)} (v_{w_\bullet \lambda} \otimes v_\mu) = 0
  \]
  for all $j \in I_\bullet$, $k \in I_\circ$.
  Hence, Lemma \ref{lem: def rel of Vi as Umod} implies that there exists a $\mathbf{U}$-module homomorphism
  \[
    V^\imath(\lambda, \mu) \rightarrow L^\imath(\lambda,\mu)
  \]
  which sends $v^\imath_{\lambda, \mu}$ to $v_{w_\bullet \lambda} \otimes v_\mu$.

  On the other hand, since $\mathbf{L}(v_{w_\bullet \lambda} \otimes v_\mu) \simeq V_\mathbf{L}(-w_\bullet \lambda, \mu)$, there exists an $\mathbf{L}$-module homomorphism
  \[
    \mathbf{L}(v_{w_\bullet \lambda} \otimes v_\mu) \rightarrow \mathbf{L} v^\imath_{\lambda, \mu}
  \]
  which sends $v_{w_\bullet \lambda} \otimes v_\mu$ to $v^\imath_{\lambda, \mu}$.

  Combining the two homomorphisms above, we see that $\mathbf{L} v^\imath_{\lambda, \mu} \simeq \mathbf{L}(v_{w_\bullet \lambda} \otimes v_\mu)$.
  By Lemma \ref{lem: Umod str is determined by Lmod str}, the $\mathbf{U}$-module structure of $V^\imath(\lambda, \mu)$ is determined by the $\mathbf{L}$-module structure of $\mathbf{L} v^\imath_{\lambda, \mu}$.
  On the other hand, by \cite[Proposition 3.4.3]{Wat23b}, the $\mathbf{U}$-module structure of $L^\imath(\lambda,\mu)$ is determined by the $\mathbf{L}$-module structure of $\mathbf{L}(v_{w_\bullet \lambda} \otimes v_\mu)$.
  Hence, the assertion follows.
\end{proof}

\begin{defi}\label{def: integrable Ui mod}
  A weight $\mathbf{U}^\imath$-module $V$ is said to be \emph{integrable} if for each weight vector $v \in V$, there exist $\lambda, \mu \in X^+$ and a $\mathbf{U}^\imath$-module homomorphism $V^\imath(\lambda, \mu) \rightarrow V$ which sends $v^\imath_{\lambda, \mu}$ to $v$.
\end{defi}

As explained in Section \ref{sec: intro}, the $\mathbf{U}$-modules $L^\imath(\lambda, \mu)$ in Proposition \ref{prop: Vi simeq Li}, which are isomorphic to $V^\imath(\lambda, \mu)$, are counterparts of $V(\lambda,\mu)$ in the theory of $\imath$canonical bases \cite{BaWa18}, \cite{BaWa21}, and the latter can be used to define integrable $\mathbf{U}$-modules.
From this point of view, Definition \ref{def: integrable Ui mod} is quite natural.

\section{Properties of integrable modules}\label{sec: int ui mod}
Although Definition \ref{def: integrable Ui mod} is a natural generalization of the definition of integrable $\mathbf{U}$-modules, it is not quite useful to determine whether a given $\mathbf{U}^\imath$-module is integrable or not.
This is because we only know presentations of the $V^\imath(\lambda,\mu)$'s as $\mathbf{U}$-modules, but not as $\mathbf{U}^\imath$-modules.
In this section, we shall give presentations as $\mathbf{U}^\imath$-modules.

\subsection{Some elements $\mathfrak{B}_{k,\zeta}^{(n)}$}
The aim of this subsection is to introduce a family $\{ \mathfrak{B}_{k,\zeta}^{(n)} \mid k \in I_\circ,\ \zeta \in X^\imath,\ n \in \mathbb{Z}_{\geq 0} \}$ of elements in $\dot{\mathbf{U}}^\imath$, which will be used to express the $\mathbf{U}^\imath$-modules $V^\imath(\lambda,\mu)$ in terms of generators and relations.

\subsubsection{Case $1:$ $\tau(k) = k = w_\bullet k$}
First, let us consider the case where $\tau(k) = k = w_\bullet k$.
For each $\zeta \in X^\imath$ and $\lambda \in X$ such that $\bar{\lambda} = \zeta$, the parity of integer $\langle h_k, \lambda \rangle$ is independent of the choice of $\lambda$.
We call it the parity of $\zeta$ at $k$.
Set $p_k(\zeta) \in \mathbb{Z}/2\mathbb{Z} = \{ \bar{0}, \bar{1} \}$ to be $\bar{0}$ if the parity of $\zeta$ at $k$ is even, and to be $\bar{1}$ otherwise.
Similarly, for each integer $n$, set $p(n)$ to be $\bar{0}$ if $n$ is even, and to be $\bar{1}$ otherwise.

For each $\zeta \in X^\imath$ and $n \geq 0$, define the element $\mathfrak{B}_{k, \zeta}^{(n)} \in \dot{\mathbf{U}}^\imath 1_\zeta$ inductively as follows:
\begin{itemize}
  \item $\mathfrak{B}_{k, \zeta}^{(0)} := 1_\zeta$.
  \item If $p_k(\zeta) = \bar{0}$, then $\mathfrak{B}_{k,\zeta}^{(1)} := (B_k - \kappa_k)1_\zeta$.
  \item If $p_k(\zeta) = p(n)$ and $n \geq 1$, then $\mathfrak{B}_{k,\zeta}^{(n)} := \frac{1}{[n]_k} B_k \mathfrak{B}_{k,\zeta}^{(n-1)}$.
  \item If $p_k(\zeta) \neq p(n)$ and $n \geq 2$, then
  \[
    \mathfrak{B}_{k, \zeta}^{(n)} := \frac{1}{[n]_k[n-1]_k}(B_k^2 - (q_k^{n-1} + q_k^{-n+1})\kappa_k B_k + (\kappa_k^2 - [n-1]_k^2 q_k \varsigma_k)) \mathfrak{B}_{k, \zeta}^{(n-2)}.
  \]
\end{itemize}

Let us see where this definition comes from.
Let $\mathbb{K}$ be an algebraic closure of the field $\mathbb{Q}(q)$, and consider the following quadratic equations
\begin{align}
  x^2 - (q_k + q_k^{-1}) a x + a^2 - b = 0 \tag*{$(\mathrm{QE})_{a,b}$} \label{eq: quadratic equation}
\end{align}
with variable $x$ and constants $a,b \in \mathbb{K}$.
Since this equation is symmetric with respect to $x$ and $a$, it follows that if $c \in \mathbb{K}$ is a solution to $(\mathrm{QE})_{a,b}$, then $a$ is a solution to $(\mathrm{QE})_{c,b}$.

Given scalars $a^{(0)}, b \in \mathbb{K}$, define a family $(a^{(n)})_{n \in \mathbb{Z}} \in \mathbb{K}^\mathbb{Z}$ of scalars as follows:
The $a^{(1)}$ and $a^{(-1)}$ are the solutions to the quadratic equation $(\mathrm{QE})_{a^{(0)},b}$; the ambiguity will not matter in the subsequent argument.
Suppose that we have defined $a^{(m)}$ for $-n \leq m \leq n$ for some $n \geq 1$.
Then, set $a^{(\pm (n+1))} := (q_k + q_k^{-1}) a^{(\pm n)} - a^{(\pm (n-1))}$.

By the construction above, it is clear that the $a^{(n+1)}$ and $a^{(n-1)}$ are the solutions to the quadratic equation $(\mathrm{QE})_{a^{(n)}, b}$ for each $n \in \mathbb{Z}$.
In particular, we have
\begin{align*}
  &a^{(n+1)} + a^{(n-1)} = (q_k + q_k^{-1}) a^{(n)}, \\
  &a^{(n+1)} a^{(n-1)} = (a^{(n)})^2 - b.
\end{align*}
Then, one can straightforwardly verify that for each $n \geq 1$, the $a^{(n)}$ and $a^{(-n)}$ are the solutions to the quadratic equation
\[
  x^2 - (q_k^n + q_k^{-n}) a^{(0)} x + (a^{(0)})^2 - [n]_k^2 b = 0.
\]

Now, let us consider the $(n+1)$-dimensional irreducible representation $\mathbf{V}_n$ of $U_q(\mathfrak{sl}_2)$, which we identify with the subalgebra of $\mathbf{U}$ generated by $E_k$, $F_k$, and $K_k^{\pm 1}$.
By a similar argument to \cite[\S 3.1]{Wat21b}, one can show that the $B_k \in \mathbf{U}^\imath$ acts on $\mathbf{V}_n$ semisimply with eigenvalues $\kappa^{(n)}, \kappa^{(n-2)}, \dots, \kappa^{(-n)}$, where we set $\kappa^{(0)} := \kappa_k$ and $b := q_k \varsigma_k$.
Therefore, we see that
\[
  \mathfrak{B}_{k,\zeta}^{(n+1)} = \frac{1}{[n]_k!}P_n(B_k) 1_\zeta \ \text{ for all } \zeta \in X^\imath \text{ with } p_k(\zeta) = p(n),
\]
where $P_n(x)$ denotes the minimal polynomial of the action of $B_k$ on $\mathbf{V}_n$.

The following are straightforward consequences of the argument above.

\begin{prop}\label{prop: frB on wt vec; type I}
  Let $V$ be a weight $\mathbf{U}$-module, $\lambda \in X$, $v \in V_\lambda$.
  Assume that $n := \langle h_k, \lambda \rangle \geq 0$ and $E_k \mathbf{L}^+ v = 0$.
  Then, we have
  \[
    \mathfrak{B}_{k, \bar{\lambda}}^{(n+1)} v = F_k^{(n+1)} v.
  \]
\end{prop}

\begin{prop}\label{prop: frB on int Umod; type I}
  Let $V$ be an integrable $\mathbf{U}$-module, $\lambda \in X$, and $v \in V_\lambda$.
  Then, there exists $N \in \mathbb{Z}_{\geq 0}$ such that $\mathfrak{B}_{k, \bar{\lambda}}^{(n+1)} v = 0$ for all $n \geq N$.
\end{prop}

\subsubsection{Case $2:$ $\tau(k) = k \neq w_\bullet k$}
Next, let us consider the case where $\tau(k) = k \neq w_\bullet k$.
For each $\zeta \in X^\imath$ and $n \geq 0$, set
\[
  \mathfrak{B}_{k, \zeta}^{(n)} := \frac{1}{[n]_k!} B_k^n 1_\zeta.
\]

Set $Y_k := B_k - F_k$ and $Z_k := F_k Y_k - q_k^{-2} Y_k F_k$.
Then, the following hold ({\it cf}.\ \cite[\S 5.2]{BaWa21}):
\begin{itemize}
  \item $Z_k$ commutes with both $F_k$ and $Y_k$.
  \item $Z_k \in \mathbf{L}^+_{w_\bullet(\alpha_k) - \alpha_k}$.
  \item $Y_k \in \mathbf{L}^+ E_k \mathbf{L}^+$.
\end{itemize}
Hence, for each $n \geq 0$, we have
\begin{align}\label{eq: expansion Bk^n}
  B_k^n = \sum_{f + 2z + y = n} a_{f,z,y} F_k^f Z_k^z Y_k^y  
\end{align}
for some $a_{f,z,y} \in \mathbb{Q}(q)$.

\begin{prop}\label{prop: frB on wt vec; type II}
  Let $V$ be a weight $\mathbf{U}$-module, $\lambda \in X$, $v \in V_\lambda$.
  Assume that $n := \langle h_k, \lambda \rangle \geq 0$ and $E_k \mathbf{L}^+ v = 0$.
  Then, we have
  \[
    \mathfrak{B}_{k, \bar{\lambda}}^{(n+1)} v = \sum_{f+2z = n+1} a_{f,z} F_k^{(f)} Z_k^z v
  \]
  for some $a_{f,z} \in \mathbb{Q}(q)$ such that $a_{n+1,0} = 1$.
\end{prop}

\begin{proof}
  The assertion is immediate from equation \eqref{eq: expansion Bk^n}.
\end{proof}

\begin{prop}\label{prop: frB on int Umod; type II}
  Let $V$ be an integrable $\mathbf{U}$-module, $\lambda \in X$, and $v \in V_\lambda$.
  Then, there exists $N \in \mathbb{Z}_{\geq 0}$ such that $\mathfrak{B}_{k, \bar{\lambda}}^{(n+1)} v = 0$ for all $n \geq N$.
\end{prop}

\begin{proof}
  Since $V$ is integrable, the element $Y_k$, which is a scalar multiple of $T_{w_\bullet}(E_{\tau(k)}) K_k^{-1}$, acts on $V$ locally nilpotently.
  In particular, there exists $y_0 \geq 0$ such that
  \[
    Y_k^{y} v = 0 \ \text{ for all } y > y_0.
  \]

  Recall that the algebra $\mathbf{L}^+$ is a quantum group of finite type.
  Hence, the $\mathbf{L}^+$-submodule of the integrable module $V$ generated by the finite set $\{ Y_k^y v \mid y \geq 0 \}$ is finite-dimensional.
  Hence, there exists $z_0 \geq 0$ such that
  \[
    Z_k^z Y_k^y v = 0 \ \text{ for all } y \geq 0,\ z > z_0.
  \]

  Since the $F_k$ acts on $V$ locally nilpotently, and the number of vectors of the form $Z_k^z Y_k^y v$ with $z,y \geq 0$ is finite, there exists $f_0 \geq 0$ such that
  \[
    F_k^f Z_k^z Y_k^y v = 0 \ \text{ for all } z,y \geq 0,\ f > f_0.
  \]

  Now, set
  \[
    N := f_0 + 2z_0 + f_0.
  \]
  By equation \eqref{eq: expansion Bk^n}, we have
  \[
    B_k^{N+1} v = \sum_{f+2z+y=N+1} a_{f,z,y} F_k^f Z_k^z Y_k^y v.
  \]
  All the summand in the right-hand side is $0$ by the definition of $N$.
  Hence, the assertion follows.
\end{proof}

\subsubsection{Case $3:$ $\tau(k) \neq k$}
Finally, let us consider the case where $\tau(k) \neq k$.
For each $\zeta \in X^\imath$ and $n > 0$, set
\[
  \mathfrak{B}_{k, \zeta}^{(n)} := \frac{1}{[n]_k!} B_k^n 1_\zeta.
\]

Set $Y_k := B_k - F_k$.
Then, the following hold ({\it cf}.\ \cite[\S 5.5]{BaWa21}):
\begin{itemize}
  \item $F_k Y_k - q_k^{-2} Y_k F_k = 0$.
  \item $Y_k \in \mathbf{L}^+ E_{\tau(k)} \mathbf{L}^+$.
\end{itemize}
Hence, for each $n \geq 0$, we have
\begin{align}\label{eq: expansion Bk^n; type III}
  B_k^n = \sum_{f + y = n} a_{f,y} F_k^f Y_k^y  
\end{align}
for some $a_{f,y} \in \mathbb{Q}(q)$.

\begin{prop}\label{prop: frB on wt vec; type III}
  Let $V$ be a weight $\mathbf{U}$-module, $\lambda \in X$, $v \in V_\lambda$.
  Assume that $n := \langle h_k, \lambda \rangle \geq 0$ and $E_k \mathbf{L}^+ v = E_{\tau(k)} \mathbf{L}^+ v = 0$.
  Then, we have
  \[
    \mathfrak{B}_{k, \bar{\lambda}}^{(n+1)} v = F_k^{(n+1)} v.
  \]
\end{prop}

\begin{proof}
  The assertion is immediate from equation \eqref{eq: expansion Bk^n; type III}.
\end{proof}

\begin{prop}\label{prop: frB on int Umod; type III}
  Let $V$ be an integrable $\mathbf{U}$-module, $\lambda \in X$, and $v \in V_\lambda$.
  Then, there exists $N \in \mathbb{Z}_{\geq 0}$ such that $\mathfrak{B}_{k, \bar{\lambda}}^{(n+1)} v = 0$ for all $n \geq N$.
\end{prop}

\begin{proof}
  One can prove the assertion in the same way as Proposition \ref{prop: frB on int Umod; type II}.
\end{proof}

\subsubsection{Summary}
Let us summarize in a unified manner what we have obtained in the preceding argument.

\begin{prop}\label{prop: summary for Bk,zeta}
  Let $V$ be a weight $\mathbf{U}$-module, $\lambda \in X$, $v \in V_\lambda$, and $k \in I_\circ$.
  Assume that $\langle h_k, \lambda \rangle \geq 0$ and $E_k \mathbf{L}^+ v = E_{\tau(k)} \mathbf{L}^+ v = 0$.
  Let $\mathbf{U}_k^-$ denote the subalgebra of $\mathbf{U}$ generated by $F_k$.
  Then, the following hold.
  \begin{enumerate}
    \item\label{item: summary for Bk,zeta 1} $\mathfrak{B}_{k,\zeta}^{(\langle h_k, \lambda \rangle + 1)} v - F_k^{(\langle h_k, \lambda \rangle + 1)} v \in \sum_{b \in \mathbf{B}_\mathbf{L}(-\infty) \setminus \{ 1 \}} \mathbf{U}^-_k F_k^{(\langle h_k, \lambda + \operatorname{wt}(b) \rangle + 1)} b v$, where we understand that $F_k^{(n)} = 0$ if $n < 0$.
    \item\label{item: summary for Bk,zeta 2} $E_k \mathfrak{B}_{k, \zeta}^{(\langle h_k, \lambda \rangle + 1)} v \in \sum_{b \in \mathbf{B}_\mathbf{L}(-\infty) \setminus \{ 1 \}} \mathbf{U} F_k^{(\langle h_k, \lambda + \operatorname{wt}(b) \rangle + 1)} b v$.
    \item\label{item: summary for Bk,zeta 3} $b' \mathfrak{B}_{k, \zeta}^{\langle h_k, \lambda \rangle + 1} v \in \sum_{\substack{b \in \mathbf{B}(-\infty) \\ \operatorname{wt}(b) \geq \operatorname{wt}(b')}} \mathbf{U}^-_k F_k^{(\langle h_k, \lambda + \operatorname{wt}(b) \rangle + 1)} bv$ for all $b' \in \mathbf{B}_\mathbf{L}(-\infty)$.
  \end{enumerate}
\end{prop}

\begin{proof}
  By Lemma \ref{lem: values in axioms for gSat}, the first assertion follows from Propositions \ref{prop: frB on wt vec; type I}, \ref{prop: frB on wt vec; type II}, and \ref{prop: frB on wt vec; type III}.
  The other assertions are immediate from the first one.
\end{proof}

\begin{prop}\label{prop: summary for Bk,zeta on int Umod}
  Let $V$ be an integrable $\mathbf{U}$-module, $\zeta \in X^\imath$, and $v \in V_\zeta$.
  Then, there exist $(a_j)_{j \in I_\bullet}, (b_j)_{j \in I_\bullet} \in \mathbb{Z}_{\geq 0}^{I_\bullet}$, $(c_k)_{k \in I_\circ} \in \mathbb{Z}_{\geq 0}^{I_\circ}$ such that
  \[
    E_j^{(a_j+1)} v = F_j^{(b_j+1)} v = \mathfrak{B}_{k,\zeta}^{(c_k+1)} v = 0.
  \]
\end{prop}

\begin{proof}
  Since the weight vector $v$ is a finite sum of weight vectors of the weight $\mathbf{U}$-module $V$ (see \eqref{eq: wt Umod as wt Uimod}), the assertion follows from Propositions \ref{prop: frB on int Umod; type I}, \ref{prop: frB on int Umod; type II}, and \ref{prop: frB on int Umod; type III}.
\end{proof}

\subsection{A Presentation of the $\mathbf{U}^\imath$-module $V^\imath(\lambda, \mu)$}
In this subsection, we fix $\lambda, \mu \in X^+$, and set
\begin{align*}
  &\nu := w_\bullet \lambda + \mu, \\
  &\mathcal{E} := \sum_{j \in I_\bullet} \mathbf{U} E_j^{(\langle h_j, -w_\bullet \lambda \rangle + 1)} m^\imath_{\lambda, \mu}, \\
  &\mathcal{E}' := \sum_{j \in I_\bullet} \mathbf{U}^\imath E_j^{(\langle h_j, -w_\bullet \lambda \rangle + 1)} m^\imath_{\lambda, \mu}, \\
  &\mathcal{F} := \sum_{j \in I_\bullet} \mathbf{U} F_j^{(\langle h_j, \mu \rangle + 1)} m^\imath_{\lambda, \mu}, \\
  &\mathcal{F}' := \sum_{j \in I_\bullet} \mathbf{U}^\imath F_j^{(\langle h_j, \mu \rangle + 1)} m^\imath_{\lambda, \mu}, \\
  &\mathcal{F}_\circ := \sum_{k \in I_\circ} \mathbf{U} F_k^{(\langle h_k, w_\bullet \lambda + \mu \rangle + 1)} m^\imath_{\lambda, \mu}, \\
  &\mathcal{B} := \sum_{k \in I_\circ} \sum_{b \in \mathbf{B}_{\mathbf{L}}(-\infty)[w_\bullet \lambda]} \mathbf{U} F_k^{(\langle h_k, \nu + \operatorname{wt}(b) \rangle + 1)} b m^\imath_{\lambda, \mu}, \\
  &\mathcal{B}' := \sum_{k \in I_\circ} \sum_{b \in \mathbf{B}_{\mathbf{L}}(-\infty)[w_\bullet \lambda]} \mathbf{U}^\imath \mathfrak{B}_{k, \overline{\nu + \operatorname{wt}(b)}}^{(\langle h_k, \nu + \operatorname{wt}(b) \rangle + 1)} b m^\imath_{\lambda, \mu}.
\end{align*}

The aim of this subsection is to show that
\[
  \mathcal{E} + \mathcal{F} + \mathcal{F}_\circ = \mathcal{E} + \mathcal{F} + \mathcal{B} = \mathcal{E}' + \mathcal{F}' + \mathcal{B}'.
\]
This provides us with a presentation of the $\mathbf{U}^\imath$-module $V^\imath(\lambda, \mu)$.

\begin{lem}\label{lem: UW = UiW}
  Let $V$ be a weight $\mathbf{U}$-module, and $W \subseteq V$ a subspace.
  If $\mathbf{U} \mathbf{R}^+ W \subseteq W$, then $\mathbf{U} W = \mathbf{U}^\imath W$.
\end{lem}

\begin{proof}
  The assertion follows from the isomorphism $\mathbf{U}^\imath 1_\lambda \simeq \dot{\mathbf{P}}^- 1_\lambda \simeq \dot{\mathbf{U}} 1_\lambda/ \mathbf{U} \mathbf{R}^+ 1_\lambda$, which is obtained from Proposition \ref{prop: Uidot simeq Pdot-} and Lemma \ref{lem: P- simeq U/UR+}.
\end{proof}

\begin{lem}\label{lem: UW = UiW for Ej}
  We have
  \[
    \mathcal{E} = \mathcal{E}'.
  \]
\end{lem}

\begin{proof}
  Let $j \in I_\bullet$ and set $v := E_j^{(\langle h_j, -w_\bullet \lambda \rangle + 1)} m^\imath_{\lambda, \mu} \in M^\imath(\lambda, \mu)$.
  Since $\mathbf{R}^+ E_j \subset \mathbf{R}^+$, we have
  \[
    \mathbf{R}^+ v \subset \mathbf{R}^+ m^\imath_{\lambda, \mu} = 0.
  \]
  Applying Lemma \ref{lem: UW = UiW} to the $1$-dimensional subspace spanned by $v$, we obtain
  \[
    \mathbf{U} v = \mathbf{U}^\imath v.
  \]
  This implies the assertion.
\end{proof}

\begin{lem}\label{lem: UW = UiW for Fj}
  We have
  \[
    \mathcal{F} = \mathcal{F}'.
  \]
\end{lem}

\begin{proof}
  Let $j \in I_\bullet$ and set $v := F_j^{(\langle h_j, \mu \rangle + 1)} m^\imath_{\lambda, \mu} \in M^\imath(\lambda, \mu)$.
  Let $x \in \mathbf{L}^+$ and $k \in I_\circ$.
  By the triangular decomposition \eqref{eq: triangular decomposition of L} of $\mathbf{L}$, the vector $E_k x v$ is, up to scalar multiplies, of the form $E_k x^- x^+ m^\imath_{\lambda, \mu}$ with $x^- \in \mathbf{L}^-$, $x^+ \in \mathbf{L}^+$.
  Since $E_k$ commutes with the elements in $\mathbf{L}^-$, we have
  \[
    E_k x^- x^+ m^\imath_{\lambda, \mu} = x^- E_k x^+ m^\imath_{\lambda, \mu} \in x^- \mathbf{R}^+ m^\imath_{\lambda, \mu} = 0.
  \]

  The argument above implies that $\mathbf{R}^+ v = 0$.
  Hence, the assertion follows by applying Lemma \ref{lem: UW = UiW} as in the proof of Lemma \ref{lem: UW = UiW for Ej}.
\end{proof}

\begin{lem}\label{lem: B + E = B' + E}
  We have
  \[
    \mathcal{B} + \mathcal{E} = \mathcal{B}' + \mathcal{E}.
  \]
\end{lem}

\begin{proof}
  For each $k \in I_\circ$ and $b \in \mathbf{B}_\mathbf{L}(-\infty)$, set
  \begin{itemize}
    \item $\lambda_b := \nu + \operatorname{wt}(b)$,
    \item $v_{k,b} := F_k^{(\langle h_k, \lambda_b \rangle + 1)} b m^\imath_{\lambda, \mu}$,
    \item $v'_{k,b} := \mathfrak{B}_{k, \overline{\lambda_b}}^{(\langle h_k, \lambda_b \rangle + 1)} b m^\imath_{\lambda, \mu}$,
    \item $W_{> b} := \sum_{k \in I_\circ} \sum_{\substack{b' \in \mathbf{B}_\mathbf{L}(-\infty) \\ \operatorname{wt}(b') > \operatorname{wt}(b)}} \mathbf{U} v_{k,b'} + \mathcal{E}$,
    \item $W'_{> b} := \sum_{k \in I_\circ} \sum_{\substack{b' \in \mathbf{B}_\mathbf{L}(-\infty) \\ \operatorname{wt}(b') > \operatorname{wt}(b)}} \mathbf{U}^\imath v'_{k,b'} + \mathcal{E}$,
    \item $W_{\geq b} := \mathbb{Q}(q) v_{k,b} + W_{> b}$,
    \item $W'_{\geq b} := \mathbb{Q}(q) v'_{k,b} + W'_{> b}$,
  \end{itemize}
  Since $b m^\imath_{\lambda,\mu} \in \mathcal{E}$ for all $b \in \mathbf{B}_\mathbf{L}(-\infty) \setminus \mathbf{B}_\mathbf{L}(-\infty)[w_\bullet\lambda]$ by equation \eqref{eq: def rel of wV(lm) in cb}, we have $v_{k,b},v'_{k,b} \in \mathcal{E}$ for all $k \in I_\circ$ and such $b \in \mathbf{B}_\mathbf{L}(-\infty)$.
  Hence, it holds that
  \[
    W_{>b} = \sum_{k \in I_\circ} \sum_{\substack{b' \in \mathbf{B}_\mathbf{L}(-\infty)[w_\bullet\lambda] \\ \operatorname{wt}(b') > \operatorname{wt}(b)}} \mathbf{U} v_{k,b'} + \mathcal{E}, \text{ and } W'_{>b} = \sum_{k \in I_\circ} \sum_{\substack{b' \in \mathbf{B}_\mathbf{L}(-\infty)[w_\bullet\lambda] \\ \operatorname{wt}(b') > \operatorname{wt}(b)}} \mathbf{U}^\imath v'_{k,b'} + \mathcal{E}.
  \]
  In particular,
  \[
    \mathbf{U} W_{\geq 1} = \mathcal{B} + \mathcal{E} \text{ and } \mathbf{U}^\imath W'_{\geq 1} = \mathcal{B}' + \mathcal{E}.
  \]

  We shall show that
  \[
    \mathbf{U}^\imath W'_{\geq b} = \mathbf{U} W_{\geq b}
  \]
  for all $b \in \mathbf{B}_\mathbf{L}(-\infty)$ by descending induction on $\operatorname{wt}(b)$.
  If this is the case, then we obtain
  \[
    \mathcal{B} + \mathcal{E} = \mathbf{U} W_{\geq 1} = \mathbf{U}^\imath W'_{\geq 1} = \mathcal{B}' + \mathcal{E},
  \]
  as desired.

  When $\operatorname{wt}(b) \gg 0$, it holds that $b' \notin \mathbf{B}_\mathbf{L}(-\infty)[w_\bullet \lambda]$ for all $b' \in \mathbf{B}_\mathbf{L}(-\infty)$ such that $\operatorname{wt}(b') \geq \operatorname{wt}(b)$.
  These imply that $v_{k,b'}, v'_{k,b'} \in \mathcal{E}$ and hence the subspaces $W_{\geq b}$ and $W'_{\geq b}$ are equal to $\mathcal{E}$.
  Hence, our claim follows in this case.

  Assume that the claim holds for all $b' \in \mathbf{B}_\mathbf{L}(-\infty)$ with $\operatorname{wt}(b') > \operatorname{wt}(b)$.
  Since
  \[
    W_{> b} = \sum_{\substack{b' \in \mathbf{B}_\mathbf{L}(-\infty) \\ \operatorname{wt}(b') > \operatorname{wt}(b)}} \mathbf{U} W_{\geq b'} \text{ and } W'_{> b} = \sum_{\substack{b' \in \mathbf{B}_\mathbf{L}(-\infty) \\ \operatorname{wt}(b') > \operatorname{wt}(b)}} \mathbf{U}^\imath W'_{\geq b'},
  \]
  these spaces are identical to each other.

  Let us show that $\mathbf{U} \mathbf{R}^+ W'_{\geq b} \subseteq W'_{\geq b}$.
  Since $W'_{\geq b} = \mathbb{Q}(q) v'_{k,b} \oplus W_{> b}$ and $\mathbf{U} W_{> b} = W_{> b}$, we only need to show that $\mathbf{R}^+ v'_{k,b} \in W_{> b}$.
  To this end, we will prove that $E_{k'} b'' v'_{k,b} \in W_{> b}$ for all $k' \in I_\circ$ and $b'' \in \mathbf{B}_\mathbf{L}(-\infty)$.
  It is easily seen that the vector $E_{k'} b'' v'_{k,b}$ is equal to $0$ unless $k' = k$.
  Hence, we may assume that $k' = k$.

  When $b'' = 1$, by Proposition \ref{prop: summary for Bk,zeta} \eqref{item: summary for Bk,zeta 2}, we have
  \[
    E_k v'_{k,b} \in \sum_{b_1 \in \mathbf{B}(-\infty) \setminus \{ 1 \}} \mathbf{U} F_k^{(\langle h_k, \nu + \operatorname{wt}(b) + \operatorname{wt}(b_1) \rangle + 1)} b_1 b m^\imath_{\lambda, \mu}.
  \]
  Since $b_1 b$ is a linear combination of $b_2 \in \mathbf{B}(-\infty)$ with $\operatorname{wt}(b_2) = \operatorname{wt}(b) + \operatorname{wt}(b_1)$, we see that
  \[
    F_k^{(\langle h_k, \nu + \operatorname{wt}(b) + \operatorname{wt}(b_1) \rangle + 1)} b_1 b m^\imath_{\lambda, \mu} \in \sum_{\operatorname{wt}(b_2) > \operatorname{wt}(b)} \mathbb{Q}(q) F_k^{(\langle h_k, \nu+\operatorname{wt}(b_2) \rangle + 1)} b_2 m^\imath_{\lambda,\mu}.
  \]
  Therefore, we obtain
  \[
    E_k v'_{k,b} \in W_{> b},
  \]
  as desired.

  When $b'' \neq 1$, by Proposition \ref{prop: summary for Bk,zeta} \eqref{item: summary for Bk,zeta 3}, we have
  \[
    E_k b'' v'_{k,b} \in \sum_{\operatorname{wt}(b') \geq \operatorname{wt}(b'')} \mathbf{U} F_k^{(\langle h_k, \nu + \operatorname{wt}(b) + \operatorname{wt}(b') \rangle + 1)} b' b m^\imath_{\lambda, \mu} = W_{> b}.
  \]
  This completes the proof of the claim that $\mathbf{U} \mathbf{R}^+ W'_{\geq b} \subseteq W'_{\geq b}$.

  Let us apply Lemma \ref{lem: UW = UiW} to obtain
  \[
    \mathbf{U} W'_{\geq b} = \mathbf{U}^\imath W'_{\geq b}.
  \]
  By Proposition \ref{prop: summary for Bk,zeta} \eqref{item: summary for Bk,zeta 1}, we have
  \[
    v'_{k,b} - v_{k,b} \in W_{> b}.
  \]
  This shows that
  \[
    W'_{\geq b} = W_{\geq b}.
  \]
  Thus, we complete the proof.
\end{proof}

\begin{lem}\label{lem: F_circ + E = B + E}
  We have
  \[
    \mathcal{F}_\circ + \mathcal{E} = \mathcal{B} + \mathcal{E}.
  \]
\end{lem}

\begin{proof}
  Clearly, we have $\mathcal{F}_\circ \subseteq \mathcal{B}$.
  Hence, we only need to show that $F_k^{(\langle h_k, \nu + \operatorname{wt}(b) \rangle + 1)} b m^\imath_{\lambda, \mu} \in \mathcal{F}_\circ$ for all $k \in I_\circ$ and $b \in \mathbf{B}_\mathbf{L}(-\infty)$.
  Since $b m^\imath_{\lambda, \mu} \in \mathcal{E}$ if $b \notin \mathbf{B}_\mathbf{L}(-\infty)[w_\bullet \lambda]$, we may assume that $b \in \mathbf{B}_\mathbf{L}(-\infty)[w_\bullet \lambda]$.
  Noting that $\operatorname{wt}(b) \in \sum_{j \in I_\bullet} \mathbb{Z}_{\geq 0} \alpha_j$ and $(\lambda - w_\bullet \lambda) - \operatorname{wt}(b) \in \sum_{j \in I_\bullet} \mathbb{Z}_{\geq 0} \alpha_j$, we see that
  \[
    \langle h_k, \operatorname{wt}(b) \rangle \leq 0 \text{ and } \langle h_k, w_\bullet \lambda + \mu + \operatorname{wt}(b) \rangle \geq \langle h_k, \lambda + \mu \rangle \geq 0.
  \]
  Therefore, the vector $F_k^{(\langle h_k, \nu + \operatorname{wt}(b) \rangle + 1)} b m^\imath_{\lambda, \mu}$ is a scalar multiple of the vector
  \[
    E_k^{-\langle h_k, \operatorname{wt}(b) \rangle} b F_k^{(\langle h_k, w_\bullet \lambda + \mu \rangle + 1)} m^\imath_{\lambda, \mu}.
  \]
  This implies our claim, and the proof is completed.
\end{proof}

\begin{thm}\label{thm: E+F+F_circ = E'+F'+B'}
  We have
  \[
    \mathcal{E} + \mathcal{F} + \mathcal{F}_\circ = \mathcal{E}' + \mathcal{F}' + \mathcal{B}'.
  \]
  In other words, the $\mathbf{U}^\imath$-module $V^\imath(\lambda, \mu)$ coincides with the quotient $\mathbf{U}^\imath$-module of $M^\imath(\lambda, \mu)$ factored by the submodule generated by
  \[
    \{ E_j^{(\langle h_j, -w_\bullet \lambda \rangle + 1)} m^\imath_{\lambda, \mu},\ F_j^{(\langle h_j, \mu \rangle + 1)} m^\imath_{\lambda, \mu} \mid j \in I_\bullet \}
  \]
  and
  \[
    \{ \mathfrak{B}_{k, \overline{w_\bullet \lambda + \mu + \mathrm{wt}(b)}}^{(\langle h_k, w_\bullet \lambda + \mu + \operatorname{wt}(b) \rangle + 1)} b m^\imath_{\lambda, \mu} \mid k \in I_\circ,\ b \in \mathbf{B}_\mathbf{L}(-\infty) \}.
  \]
\end{thm}

\begin{proof}
  The assertion follows from Lemmas \ref{lem: UW = UiW for Ej}, \ref{lem: UW = UiW for Fj}, \ref{lem: B + E = B' + E}, and \ref{lem: F_circ + E = B + E}.
\end{proof}

\begin{cor}
  Let $V$ be an integrable $\mathbf{U}^\imath$-module.
  Then, for each $\zeta \in X^\imath$ and $v \in V_\zeta$, there exist $(a_j)_{j \in I_\bullet}, (b_j)_{j \in I_\bullet} \in \mathbb{Z}_{\geq 0}^{I_\bullet}$ and $(c_k)_{k \in I_\circ} \in \mathbb{Z}_{\geq 0}^{I_\circ}$ such that
  \[
    E_j^{(a_j+1)} v = F_j^{(b_j+1)} v = 0 \ \text{ for all } j \in I_\bullet
  \]
  and
  \[
    \mathfrak{B}_{k,\zeta}^{(c_k+1)} v = 0 \ \text{ for all } k \in I_\circ.
  \]
\end{cor}

\subsection{Some sufficient conditions}
In this subsection, we give some sufficient conditions for a weight $\mathbf{U}^\imath$-module to be integrable.

\begin{prop}\label{prop: local nilp => integrable}
  Let $V$ be a weight $\mathbf{U}^\imath$-module generated by a weight vector $v$ of weight $\zeta$.
  Suppose that there exist $(a_j)_{j \in I_\bullet}, (b_j)_{j \in I_\bullet} \in \mathbb{Z}_{\geq 0}^{I_\bullet}$ and $(c_k)_{k \in I_\circ} \in \mathbb{Z}_{\geq 0}^{I_\circ}$ such that
  \[
    E_j^{(a_j+1)} v = F_j^{(b_j+1)} v = 0 \ \text{ for all } j \in I_\bullet
  \]
  and
  \[
    \mathfrak{B}_{k,\zeta}^{(c_k+1)} v = 0 \ \text{ for all } k \in I_\circ.
  \]
  Then, $V$ is integrable.
\end{prop}

\begin{proof}
  The assumption on $v$ implies that the $\mathbf{L}$-module $\mathbf{L} v$ is integrable.
  Since $\mathbf{L}$ is a quantum group of finite type, every cyclic integrable module is finite-dimensional.
  Therefore, the set
  \[
    B_V := \{ b \in \mathbf{B}_\mathbf{L}(-\infty) \mid bv \neq 0 \}
  \]
  is finite.

  Let $\lambda,\mu \in X^+$ be such that $\overline{w_\bullet \lambda + \mu} = \zeta$.
  By replacing $\lambda$ and $\mu$ by $\lambda + \tau(\nu)$ and $\mu + \nu$, respectively for a suitable $\nu \in X^+$, we may assume that
  \begin{align*}
    &\langle h_j, -w_\bullet \lambda \rangle \geq a_j, \\
    &\langle h_j, \mu \rangle \geq b_j, \\
    &\langle h_k, w_\bullet \lambda + \mu + \operatorname{wt}(b) \rangle \geq c_k
  \end{align*}
  for all $j \in I_\bullet$, $k \in I_\circ$, and $b \in B_V$.
  These imply that
  \begin{align*}
    &E_j^{(\langle h_j, -w_\bullet \lambda \rangle + 1)} v \in \mathbf{U}^\imath E_j^{(a_j+1)}v = 0, \\
    &F_j^{(\langle h_j, \mu \rangle + 1)} v \in \mathbf{U}^\imath F_j^{(b_j+1)}v = 0, \\
    &\mathfrak{B}_{k, \overline{w_\bullet\lambda + \mu} + \mathrm{wt}^\imath(b)}^{(\langle h_k, w_\bullet\lambda + \mu + \operatorname{wt}(b) \rangle + 1)} bv = b \mathfrak{B}_{k, \overline{w_\bullet\lambda + \mu}}^{(\langle h_k, w_\bullet\lambda + \mu + \operatorname{wt}(b) \rangle + 1)} v \in \mathbf{U}^\imath \mathfrak{B}_{k,\zeta}^{(c_k+1)} v = 0
  \end{align*}
  for all $j \in I_\bullet$, $k \in I_\circ$, and $b \in B_V$.
  Hence, the assertion follows from Theorem \ref{thm: E+F+F_circ = E'+F'+B'}.
\end{proof}

\begin{prop}\label{prop: int Umod is int Uimod}
  An integrable $\mathbf{U}$-module is integrable as a $\mathbf{U}^\imath$-module.
\end{prop}

\begin{proof}
  Let $V$ be an integrable $\mathbf{U}$-module.
  It is a weight $\mathbf{U}^\imath$-module.
  Let $\zeta \in X^\imath$ and $v \in V_\zeta$.
  Let us write
  \[
    v = \sum_{r=1}^s v_r
  \]
  for some $s \in \mathbb{Z}_{> 0}$, $v_r \in V_{\lambda_r}$, and $\lambda_r \in X$ with $\overline{\lambda_r} = \zeta$.

  Since $V$ is an integrable $\mathbf{U}$-module, there exist $a_j,b_j \in \mathbb{Z}_{\geq 0}$ for $j \in I_\bullet$ such that
  \[
    E_j^{(a_j+1)} v_r = F_j^{(b_j+1)} v_r = 0
  \]
  for all $r = 1,\dots,s$.
  Hence, we obtain
  \[
    E_j^{(a_j+1)} v = F_j^{(a_j+1)} v = 0 \ \text{ for all } j \in I_\bullet.
  \]

  Also, for each $k \in I_\circ$, by Proposition \ref{prop: summary for Bk,zeta on int Umod}, there exists $c_k \geq 0$ such that
  \[
    \mathfrak{B}_{k,\zeta}^{(c_k+1)} b v = 0 \ \text{ for all } b \in \mathbf{B}_\mathbf{L}(-\infty).
  \]
  This completes the proof.
\end{proof}

\section{Quantized coordinate algebras}\label{sec: qca}
In this section, we assume the following:
\begin{itemize}
  \item the Cartan datum $I$ is of finite type,
  \item the involution $\rho$ on $\mathbf{U}$ (\cite[19.1.1]{Lus93}) given by 
  \[
    \rho(E_i) = q_i K_i F_i, \quad \rho(F_i) = q_i K_i^{-1} E_i, \quad \rho(K_h) = K_h,
  \]
  preserves $\mathbf{U}^\imath$,
  \item the canonical basis $\dot{\mathbf{B}}^\imath$ of $\dot{\mathbf{U}}^\imath$ possesses the stability property (or strong compatibility) ({\it cf}.\ \cite[Remark 6.18]{BaWa18}).
\end{itemize}
For example, if $\kappa_i$ are all zero and each $\varsigma_i$ is a certain signed power of $q_i$ as in \cite[proof of Lemma 4.2.1]{Wat23e}, then the last two conditions above are satisfied by \cite[Proposition 4.6]{BaWa18} and \cite[Theorem 4.3.1]{Wat23e}.
The parameters in \cite[proof of Lemma 4.2.1]{Wat23e} are almost the same as in \cite[Table 3]{BaWa18}, where the constants are deduced by requiring that each $\varsigma_i$ is a signed power of $q_i$ and that the bar-involution (which was used to construct $\dot{\mathbf{B}}^\imath$) on $\mathbf{U}^\imath$ exists.

\subsection{Simple integrable modules}
Let $\lambda, \mu \in X^+$ and consider the $\mathbf{U}^\imath$-module $V^\imath(\lambda, \mu)$.
It is equipped with a nondegenerate symmetric bilinear form $(,) = (,)_{\lambda, \mu}$ such that $(v^\imath_{\lambda, \mu}, v^\imath_{\lambda, \mu}) = 1$ and
\[
  (x u, v) = (u, \rho(x) v) \ \text{ for all } x \in \mathbf{U}^\imath,\ u,v \in V^\imath(\lambda, \mu)
\]
({\it cf}.\ \cite[\S 6.6]{BaWa18}).
Since $V^\imath(\lambda, \mu)$ is finite-dimensional, the existence of such a bilinear form ensures that it is semisimple.
In fact, if $W$ is a submodule of $V^\imath(\lambda,\mu)$, and $W^\perp$ denotes the orthogonal complement of $W$ with respect to $(,)_{\lambda,\mu}$, then $W^\perp$ is also a submodule of $V^\imath(\lambda,\mu)$ since $\rho$ preserves $\mathbf{U}^\imath$.
With some index set $\Gamma$, let
\[
  \{ V(\gamma) \mid \gamma \in \Gamma \}
\]
denote a complete set of representatives for the isomorphism classes of simple $\mathbf{U}^\imath$-modules appearing as a submodule of $V^\imath(\lambda, \mu)$ for some $\lambda, \mu \in X^+$.
Since the $\mathbf{U}^\imath$-module $V^\imath(\lambda,\mu)$ is integrable by Proposition \ref{prop: int Umod is int Uimod}, so is each $V(\gamma)$.

Let $V$ be a simple integrable $\mathbf{U}^\imath$-module, and $v \in V$ a nonzero weight vector.
Then, there exist $\lambda, \mu \in X^+$ and a $\mathbf{U}^\imath$-module homomorphism $V^\imath(\lambda, \mu) \rightarrow V$ which sends $v^\imath_{\lambda, \mu}$ to $v$.
Since $V$ is simple, this homomorphism must be surjective.
The semisimplicity of $V^\imath(\lambda, \mu)$ implies that it has a simple submodule isomorphic to $V$.
Therefore, $V$ is isomorphic to $V(\gamma)$ for some $\gamma \in \Gamma$.

The argument above shows that the set $\{ V(\gamma) \mid \gamma \in \Gamma \}$ is a complete set of representatives for the isomorphism classes of simple integrable $\mathbf{U}^\imath$-modules.

\begin{lem}\label{lem: almost all Bidot annihilate Vgam}
  Let $\gamma \in \Gamma$.
  Then, almost all the elements in $\dot{\mathbf{B}}^\imath$ acts on the $\mathbf{U}^\imath$-module $V(\gamma)$ as $0$:
  \[
    \sharp \{ b \in \dot{\mathbf{B}}^\imath \mid b V(\gamma) \neq 0 \} < \infty.
  \]
\end{lem}

\begin{proof}
  Let $B(\gamma) = \{ v^\gamma_i \mid i \in I_\gamma \}$ be a basis of $V(\gamma)$ consisting of weight vectors.
  Since $V(\gamma)$ is finite-dimensional, the index set $I_\gamma$ is finite.
  By the integrability of $V(\gamma)$, for each $i \in I_\gamma$, we may take $\lambda_i,\mu_i \in X^+$ such that there exists a $\mathbf{U}^\imath$-module homomorphism $V^\imath(\lambda_i,\mu_i) \rightarrow V(\gamma)$ which sends $v^\imath_{\lambda_i,\mu_i}$ to $v^\gamma_i$.

  Since the $\dot{\mathbf{B}}^\imath$ possesses the stability property, we have
  \[
    \sharp \{ b \in \dot{\mathbf{B}}^\imath \mid b v^\imath_{\lambda,\mu} \neq 0 \} = \dim V^\imath(\lambda,\mu) \ \text{ for all } \lambda,\mu \in X^+.
  \]
  Hence, the number of elements $b \in \dot{\mathbf{B}}^\imath$ such that $b v^\imath_{\lambda_i,\mu_i} \neq 0$ for some $i \in I_\gamma$ is finite.
  
  Now, suppose that $b \in \dot{\mathbf{B}}^\imath$ acts on $V(\gamma)$ as nonzero.
  Then, there exists $i \in I_\gamma$ such that $b v^\imath_{\lambda_i,\mu_i} \neq 0$.
  By the argument above, the number of such elements is finite.
\end{proof}

\subsection{Space of matrix coefficients}
Given a linear space $V$, let $V^*$ denote the linear dual space $\operatorname{Hom}_{\mathbb{Q}(q)}(V, \mathbb{Q}(q))$, and $\langle \ \mid \ \rangle \colon V \times V^* \rightarrow \mathbb{Q}(q)$ the canonical pairing.

Set $\dot{\mathbf{U}}^{\imath *} := (\dot{\mathbf{U}}^\imath)^*$.
It is equipped with the following $\mathbf{U}^\imath$-bimodule structure:
\[
  \langle u \mid x \phi y \rangle = \langle y u x \mid \phi \rangle \ \text{ for all } x,y \in \mathbf{U}^\imath,\ u \in \dot{\mathbf{U}}^\imath,\ \phi \in \dot{\mathbf{U}}^{\imath *}.
\]

For each $\gamma \in \Gamma$, $v \in V(\gamma)$, and $\phi \in V(\gamma)^*$, let $c^\gamma_{v, \phi} \in \dot{\mathbf{U}}^{\imath *}$ denote the corresponding matrix coefficient:
\[
  \langle x \mid c^\gamma_{v, \phi} \rangle = \langle xv \mid \phi \rangle \ \text{ for all } x \in \dot{\mathbf{U}}^\imath.
\]
The matrix coefficients $\{ c^\gamma_{v,\phi} \mid v \in V(\gamma),\ \phi \in V(\gamma)^* \}$ form a simple $\mathbf{U}^\imath$-bimodule isomorphic to $V(\gamma) \otimes V(\gamma)^*$; the matrix coefficient $c^\gamma_{v, \phi}$ corresponds to the vector $v \otimes \phi$.

Let $\mathbf{O}^\imath$ denote the subspace of $\dot{\mathbf{U}}^{\imath *}$ consisting of $\phi \in \dot{\mathbf{U}}^{\imath *}$ such that the $\mathbf{U}^\imath$-submodule $\mathbf{U}^\imath \phi$ and the right $\mathbf{U}^\imath$-submodule $\phi \mathbf{U}^\imath$ are integrable.

\begin{lem}
  Every integrable cyclic $\mathbf{U}^\imath$-module is finite-dimensional.
\end{lem}

\begin{proof}
  Let $V$ be an integrable cyclic $\mathbf{U}^\imath$-module with a cyclic vector $v$.
  Without any loss of generality, we may assume that $v$ is a weight vector.
  By the definition of integrable $\mathbf{U}^\imath$-modules, there exist $\lambda,\mu \in X^+$ and a $\mathbf{U}^\imath$-module homomorphism $V^\imath(\lambda,\mu) \rightarrow V$ which sends $v^\imath_{\lambda,\mu}$ to $v$.
  Since $V$ is cyclic, the homomorphism must be surjective.
  Hence, the finite-dimensionality of $V^\imath(\lambda,\mu)$ implies that of $V$.
  Thus, we complete the proof.
\end{proof}

The following two statements can be proved by a similar way to classical results (e.g. see \cite[6.2.6]{Pro07}).

\begin{lem}
  Let $\phi \in \dot{\mathbf{U}}^{\imath *}$.
  Then, the following are equivalent:
  \begin{itemize}
    \item $\mathbf{U}^\imath \phi$ is integrable.
    \item $\phi \mathbf{U}^\imath$ is integrable.
  \end{itemize}
\end{lem}

\begin{prop}\label{prop: Peter-Weyl type decomp}
  The $\mathbf{O}^\imath$ is a $\mathbf{U}^\imath$-bisubmodule of $\dot{\mathbf{U}}^{\imath *}$.
  Moreover, the assignment
  \[
    V(\gamma) \otimes V(\gamma)^* \rightarrow \dot{\mathbf{U}}^{\imath *};\ v \otimes \phi \mapsto c^\gamma_{v,\phi}
  \]
  gives rise to a Peter-Weyl type decomposition:
  \[
    \bigoplus_{\gamma \in \Gamma} V(\gamma) \otimes V(\gamma)^* \simeq \mathbf{O}^\imath \ (\text{as $\mathbf{U}^\imath$-bimodules}).
  \]
\end{prop}

Let $\dot{\mathbf{B}}^{\imath *} := \{ \delta_b \mid b \in \dot{\mathbf{B}}^\imath \}$ denote the dual basis to $\dot{\mathbf{B}}^\imath$:
\[
  \langle b_1 \mid \delta_{b_2} \rangle = \delta_{b_1, b_2} \ \text{ for all } b_1,b_2 \in \dot{\mathbf{B}}^\imath.
\]

\begin{thm}\label{thm: dual icb spans Oi}
  The $\dot{\mathbf{B}}^{\imath *}$ is a linear basis of $\mathbf{O}^\imath$.
\end{thm}

\begin{proof}
  First, we shall show that $\dot{\mathbf{B}}^{\imath *} \subset \mathbf{O}^\imath$.

  Let $b \in \dot{\mathbf{B}}^\imath$.
  Take $\lambda, \mu \in X^+$ in a way such that $b v^\imath_{\lambda, \mu} \neq 0$.
  Set
  \[
    B := \{ b' \in \dot{\mathbf{B}}^\imath \mid b' v^\imath_{\lambda, \mu} \neq 0 \}.
  \]
  Note that the set $B$ is finite.
  In fact, since the $\mathbf{U}^\imath$-module $V^\imath(\lambda, \mu)$ is integrable, there exist $(a_j)_{j \in I_\bullet}, (b_j)_{j \in I_\bullet} \in \mathbb{Z}_{\geq 0}^{I_\bullet}$, $(c_k)_{k \in I_\circ} \in \mathbb{Z}_{\geq 0}^{I_\circ}$ such that
  \[
    E_j^{(a_j+1)} b' v^\imath_{\lambda,\mu} = F_j^{(b_j+1)} b' v^\imath_{\lambda,\mu} = \mathfrak{B}_{k, \zeta}^{(c_k+1)} b' v^\imath_{\lambda,\mu} \ \text{ for all } j \in I_\bullet,\ k \in I_\circ,\ b' \in B,
  \]
  where $\zeta \in X^\imath$ denotes the weight of $b' v^\imath_{\lambda, \mu}$.

  We shall show that $\delta_b x = 0$ for all
  \[
    x \in \{ E_j^{(a_j+1)},\ F_j^{(b_j+1)},\ \mathfrak{B}_{k, \zeta}^{(c_k+1)} \mid j \in I_\bullet,\ k \in I_\circ,\ \zeta \in X^\imath \}.
  \]
  Let $b_1 \in \dot{\mathbf{B}}^\imath$ and write
  \begin{align}\label{eq: expansion of xb1}
    x b_1 = \sum_{b_2 \in \dot{\mathbf{B}}^\imath} c_{b_2} b_2
  \end{align}
  with $c_{b_2} \in \mathbb{Z}[q,q^{-1}]$.
  Then, we have
  \[
    \langle b_1 \mid \delta_b x \rangle = \langle x b_1 \mid \delta_b \rangle = c_b.
  \]
  We want to show that $c_b = 0$.
  By equation \eqref{eq: expansion of xb1}, we have
  \[
    x b_1 v^\imath_{\lambda, \mu} = \sum_{b_2 \in B} c_{b_2} b_2 v^\imath_{\lambda, \mu}.
  \]
  The left-hand side equals $0$ by the definition of $x$.
  Since the elements $\{ b_2 v^\imath_{\lambda,\mu} \mid b_2 \in B \}$ forms a basis of $V^\imath(\lambda,\mu)$, we must have $c_{b_2} = 0$ for all $b_2 \in B$.
  In particular, we obtain $c_b = 0$.
  
  Next, we shall prove that the $\dot{\mathbf{B}}^{\imath *}$ spans $\mathbf{O}^\imath$.
  Let $\{ v^\gamma_i \mid i \in I_\gamma \}$ be a linear basis of $V(\gamma)$ and $\{ \phi^\gamma_i \mid i \in I_\gamma \}$ its dual basis.
  By the Peter-Weyl type decomposition \ref{prop: Peter-Weyl type decomp}, $\mathbf{O}^\imath$ has the following linear basis:
  \[
    \{ c^\gamma_{v^\gamma_i, \phi^\gamma_j} \mid i,j \in I_\gamma \}.
  \]
  We can formally express these basis elements as
  \[
    c^\gamma_{v^\gamma_i, \phi^\gamma_j} = \sum_{b \in \dot{\mathbf{B}}^\imath} c^\gamma_{v^\gamma_i, \phi^\gamma_j}(b) \delta_b.
  \]
  Since $c^\gamma_{v^\gamma_i,\phi^\gamma_j}(b) = 0$ for almost all $b \in \dot{\mathbf{B}}^\imath$ (Lemma \ref{lem: almost all Bidot annihilate Vgam}), the right-hand side is a finite sum.
  This proves the assertion.
\end{proof}

\begin{cor}
  The space of matrix coefficients for the semisimple integrable $\mathbf{U}^\imath$-modules coincides with the quantum coordinate algebra of $\dot{\mathbf{U}}^\imath$ in \cite[3.2.2]{BaSo22}.
\end{cor}

\begin{proof}
  The former is the $\mathbf{O}^\imath$, and the latter is the subspace of $\dot{\mathbf{U}}^{\imath *}$ spanned by the dual basis $\dot{\mathbf{B}}^{\imath *}$.
  Hence, the assertion follows from Theorem \ref{thm: dual icb spans Oi}.
\end{proof}

\section*{Declaration}
The author has no conflicts of interest to declare.


\end{document}